\documentclass[11pt]{amsart}

 \usepackage{graphicx} 
 \usepackage{amscd} 
 \usepackage[colorlinks,linkcolor={blue}]{hyperref}
 \usepackage{xcolor}

 \def\a{\alpha}
 
 \def\be{\beta}
 
 \def\de{\delta}
 
 \def\e{\varepsilon}

 \def\ga{\gamma}

 \def\tga{{\tilde{\gamma}}}

 \def\la{\lambda}

 \def\si{\sigma}
 
 \def\th{\theta}
 \def\om{\omega}

 \def\re{{\mathbb R}}
 \def\na{{\mathbb N}}

 \def\then{\Longrightarrow}
 \def\ov{\overline}
 \def\Z{{\mathbb Z}}

 \def\A{{\mathbb A}}

 \def\D{{\mathbb D}}

 \def\E{{\mathbb E}}
 
 \def\cE{{\mathcal E}}
 
 \def\F{{\mathbb F}}
 
 \def\tf{{\tilde f}}

 \def\cG{{\mathcal G}}

 \def\cL{{\mathcal L}}

 \def\cM{{\mathcal M}}

 \def\cO{{\mathcal O}}
 \def\P{{\mathbb P}}
 \def\cP{{\mathcal P}}
 
 \def\SS{{\mathbb S}}
 \def\T{{\mathbb T}}

 \def\cU{{\mathcal U}}
 
 \def\tf{{\tilde{f}}}
 
 \def\tnu{{\tilde{\nu}}}

 \def\cW{{\mathcal W}}
  
 \def\tq{{\tilde{q}}}
 
 \def\tq1{{\tilde{q}_1}}

 \def\X{{\mathbb X}}
 \def\Y{{\mathbb Y}}

 \def\ty{{\tilde{y}}}

 \definecolor{dgreen}{rgb}{0,0.3,0}
 \definecolor{dred}{rgb}{0.8,0,0}

 \def \lv{\left\vert}
 \def \rv{\right\vert}
 \def \lV{\left\Vert}
 \def \rV{\right\Vert}
 \def \ov{\overline}
 
 \def \then{\Longrightarrow}

 \DeclareMathOperator{\supp}{supp}
 
 \DeclareMathOperator{\diam}{diam}

  \renewcommand{\proofname}{{\bf Proof:}}

 \theoremstyle{plain}

 \newtheorem{MainThm}{Theorem}
 \newtheorem{MainCor}[MainThm]{Corollary}

  \swapnumbers

 \newtheorem{Thm}{Theorem}[section]
 
 \newtheorem{Lemma}[Thm]{\bf Lemma}
 
 \newtheorem{Corollary}[Thm]{\bf Corollary}
 \newtheorem{Theorem}[Thm]{\bf Theorem}
 \newtheorem{Proposition}[Thm]{\bf Proposition}

\newtheorem{Statement}[Thm]{\bf Statement}

 \theoremstyle{definition}

 \theoremstyle{remark}
 \newtheorem{Remark}[Thm]{\bf Remark}

 \newtheoremstyle{Cl}
  {5pt}
  {3pt}
  {\sl}
  {}
  {\it}
  {:}
  {.5em}
  {}

 \theoremstyle{Cl}
 \newtheorem*{Claim}{Claim}
 \newtheorem{claim}[Thm]{Claim}

 \def\begincproof{
                  \renewcommand{\proofname}{\it Proof:}
                  \begin{proof}
                 }

 \def\endcproof{
                \renewcommand{\qedsymbol}{$\diamondsuit$}
                \end{proof} 
                \renewcommand{\qedsymbol}{\openbox}
                \renewcommand{\proofname}{\bf Proof:}
               }

\def\t{T}
\def\ti{S}

\def\oF{\ov{F}}

\def\oy{\ov{y}}

\DeclareMathOperator\Lip{Lip}
\DeclareMathOperator\Per{Per}
\DeclareMathOperator\Fix{Fix}

\usepackage[colorinlistoftodos]{todonotes}
\makeatletter
\providecommand\@dotsep{5}
\makeatother
 \renewcommand{\proofname}{{\bf Proof:}}

 \title
 {Ground States are generically a periodic orbit}

 \author[G. Contreras]{Gonzalo Contreras}

\address{CIMAT \\
          A.P. 402, 36.000 \\
          Guanajuato. GTO \\
          M\'exico.}
\email{gonzalo@cimat.mx}
\thanks{Gonzalo Contreras was Partially supported by {\sc conacyt}, Mexico, grant  178838.}

\begin{document}

   \begin{abstract} We prove that for an expanding transformation the maximizing measures
  	                     of a generic Lipschitz function are supported on a single periodic orbit.
   \end{abstract}

\openup 2pt

 \maketitle
 
 \section{Introduction}

 Let $X$ be a compact metric space and $T:X\to X$ an expanding map. 
 This means that $T$ is Lipschitz continuous and there are numbers
 $d\in\Z^+$, $0<\la<1$ such that for every point $x\in X$ there is a 
 neighborhood $U_x$ of $x$ in $X$ and  continuous branches $S_i$,
 $i=1,\ldots,\ell_x\le d$ of the inverse of $T$ with disjoint images $S_i(U_x)$,
 such that 
 $T^{-1}(U_x)=\bigcup_{i=1}^{\ell_x} S_i(U_x)$,  $T\circ S_i= I_{U_x}$~$\forall i$,
 and
$$
d\big(S_i(y),S_i(z)\big)\le \la\, d(y,z)
\qquad 
\forall y,  z\in U_x.
$$

  Given a continuous function $F:X\to\re$, a {\it maximizing measure}
  is a  $T$-invariant Borel probability measure $\mu$ which 
  maximizes the integral of $F$ among all $T$-invariant Borel probabilities:
  $$
  \int F\;d\mu = \sup\Big\{\int F\;d\nu \;\Big|\; \nu\in\cM(T)\Big\},
  $$
  where
  $$
  \cM(T)=\big\{\text{ $T$-invariant Borel probabilities in $X$ }\big\}.
  $$
  
  Recall that an {\it equilibrium state} for $F$ is an invariant Borel probability $\mu_F$ which satisfies
  $$
  \mu_F :=  \arg\max\Big\{h_\mu(T)+\int F\;d\mu\; \Big|\;\mu\in\cM(T)\,\Big\}.
  $$
   {\it Ground states} are the  zero temperature limits of equilibrium states.
  This means limits of the form $\lim\limits_{\be\to+\infty} \mu_{\be F}$.
  Here $\beta$ is 
  interpreted as the inverse of the temperature. It is known \cite[Proposition 29]{CLT}
  that if the limit of a sequence $\{\mu_{\be_k F}\}_k$ with $\be_k\to\infty$ exists, then it
  has to be a maximizing measure with maximal entropy among the maximizing measures.
   Br\'emont \cite{Bre1} proves that the limit  $\lim\limits_{\be\to +\infty}\mu_{\be F}$ exists if $F$ is locally constant.
  Chazottes, Gambaudo and Ugalde \cite{CGU} give a characterization of the limit and a new proof
  of Br\'emont's result. 
  Leplaideur \cite{Lep} gives another proof of Br\'emont's theorem and a generalization:
  if $G$ is H\"older continuous and $F$ is locally constant, then 
  the limit when $\be\to\infty$  of the equilibrium states
  of $G+\be\,F$  exists. 
    Chazottes and Hochman \cite{ChHo} give an example of a Lipschitz function $F$ 
  for which the zero temperature limit does not exist.
  An example with a discontinuous function 
  was given before by Van Enter and Ruszel \cite{vER}.

  For generic H\"older or Lipschitz functions $F$, the maximizing measure is unique.
  This is proven in Contreras, Lopes, Thieullen~\cite{CLT} 
  and it is presented in a general version in 
  Jenkinson~\cite{ErgOpt}.
  The ideas came from an analogous result for lagrangian systems by Ma\~n\'e~\cite{Ma6}.
  After Jenkinson lecture notes \cite{ErgOpt} the study of maximizing measures for a fixed
  dynamical system became known 
  as Ergodic Optimization. Surveys of the subject are presented by Jenkinson \cite{ErgOpt} 
  and Baraviera, Leplaideur, Lopes \cite{BLL}. 
  
  \begin{Theorem}[Contreras, Lopes, Thieullen \cite{CLT}, see also Jenkinson~\cite{ErgOpt}]\label{CLT}
  Let $T:X\to X$ be a continuous map of a compact metric space. 
  Let $E$ be a topological vector space which is densely and continuously embedded in
  $C^0(X,\re)$. Write
  $$
  \cU(E):=\big\{\,F\in E \;\big|\;\text{\rm there is a unique $F$-maximizing measure}\;\big\}.
  $$
  Then $\cU(E)$ is a countable intersection of open and dense sets.
  
  If moreover $E$ is a Baire space, then $\cU(E)$ is dense in $E$.
  \end{Theorem}
  
  The main conjecture in Ergodic Optimization during the last decade have been
  wether the maximizing measure for generic H\"older or Lipschitz functions $F$
  is supported on a periodic orbit. For lagrangian systems an analogous statement 
  is known as Ma\~n\'e's conjecture.

   On the space $\Lip(X,\re)$ of Lipschitz functions on $X$ we use the norm
   \begin{equation}\label{lipnorm}
   \lV f\rV := \sup_{x\in X}|f(x)| + \sup_{x\ne y}\frac{|f(x)-f(y)|}{d(x,y)}.
   \end{equation}
   We denote the the first term in \eqref{lipnorm} as $\lV f \rV_0$ 
   and the second term as $\Lip(f)$.
   
    Here we prove
  \begin{MainThm}\label{ThmA}
  If $X$ is a compact metric space and $T:X\hookleftarrow$ is an expanding map 
  then there is an open and dense set $\cO\subset \Lip(X,\re)$ such that for all 
  $F\in\cO$ there is a single $F$-maximizing measure and it is supported on 
  a periodic orbit.
  \end{MainThm}
  
  \begin{MainCor}\label{CorB}
  For an open and dense set $\cO$ of Lipschitz functions $F$ on $X$ the
  zero temperature limit $\lim\limits_{\be\to+\infty}\mu_{\be F}$ exists and 
  it is supported on a 
  single periodic orbit.
  \end{MainCor}

  On the negative side, for expanding transformations 
  Bousch \cite[Proposition 9, p. 306]{Bou1} proves that for generic continuous functions
  the maximizing measure is not supported on a periodic orbit. 
  Indeed, its support is the total space \cite[Rem. 7]{Bou1}.
  Bousch theorem in he case of hyperbolic sets is presented by 
  Jenkinson in \cite[Theorem 4.2]{ErgOpt}.
  
  There have been several approaches to the conjecture from which 
  we will use some of their techniques. Write
  $$
  \cP(E) := \big\{\, F\in E\; \big|\; \text{\rm the unique $F$-maximizing measure is supported 
  on a periodic orbit }\big\}. 
  $$
  Contreras, Lopes, Thieullen \cite{CLT} prove that $\cP(E)$ is open for $E=C^\a(X,\re)$ 
  the space of $\a$-H\"older continuous functions and in the $\a$-H\"older topology
   it is open and dense in $E=C^{!\a}(X,\re)$,
  the space of functions $F:X\to\re$ such that
  $$
  \forall \eta>0 \quad \exists\e>0\qquad d(x,y)<\e \implies |F(x)-F(y)|<\eta \, d(x,y)^\a.
  $$
  The main technique is the introduction of a sub-action $u:X\to\re$ to transform the function $F$
  to a cohomologous function $G= F+u-u\circ T$ such that $G\le a = \int G \, d\mu^G$,
  where $\mu^G$ is a maximizing measure for $G$ and $F$. The sub-action is defined 
  similarly, and plays the same role, as a sub-solution of the Hamilton-Jacobi equation 
  for Lagrangian systems. In fact analogous constructions to the weak KAM theory 
  can be translated to this setting. In proposition~\ref{wK} we construct a sub-action
  following the original method by Fathi \cite{Fa1} to construct weak KAM solutions.  
  This method was used in ergodic optimization by Bousch in \cite{BouschP}.
  In fact many results from Lagrangians systems can be translated to the ergodic 
  optimization setting, see for example Garibaldi, Lopes, Thieullen~\cite{GLT}.

  Bousch proves that $\cP(E)$ is dense for Walters functions. 
  Yuan and Hunt \cite{YH} prove that if a fixed measure is maximizing for an open set
  of functions $F$ in the Lipschitz topology, then it is supported on a periodic orbit.
  Their method of perturbation is the basis of the present work. 
  Quas and Siefken \cite{QS} work  in a one-sided shift. They prove that $\cP(E)$ 
  contains an open and dense set if $E$ is the space of super-continuous functions.
  They present an elegant version of the method of Yuan and Hunt. 
  We need to modify it for Lipschitz functions and pseudo-orbits with finitely many jumps
  in Proposition~\ref{perturb}.
  
  Another ingredient of the proof is the following theorem.
  As a weak version of the conjecture, Morris \cite{Morris} proves
  \begin{Theorem}[Morris \cite{Morris}]\label{Tmorris}
  Let $X$ be a compact metric space and $T:X\hookleftarrow$ an expanding map.
  There is a residual set $\cG\subset \Lip(X,\re)$ such that if
  $F\in \cG$ then there is a unique $F$-maximizing measure and 
  it has zero metric entropy.
  \end{Theorem}
 
  The idea of the proof of Theorem~\ref{Tmorris}
  is to use a periodic orbit with small action 
  and small period constructed by Bressaud and Quas \cite{BQ} and perturb $F$
  so that the new minimizing measures are nearby the periodic orbit 
  and hence have small entropy.
  
  The original version of  Theorem~\ref{Tmorris} is for H\"older functions in a shift of finite type. 
  In appendix~\ref{aze} we describe the modifications from the proof in Morris~\cite{Morris} 
  needed to obtain Theorem~\ref{Tmorris}.

 \bigskip
 
 In section~\ref{S1} we develop the techniques from ergodic optimization 
 that we need and present the main perturbation result in proposition~\ref{perturb}.
 In section~\ref{S2} we prove Theorem~\ref{ThmA} with an argument by contradiction.
 We show that if the conditions for a perturbation as in proposition~\ref{perturb}
 do not hold then the entropy must be positive, contradicting Morris Theorem~\ref{Tmorris}.

\section{Preliminars}\label{S1}
  Since $X$ is compact there is a finite subcover of $\{U_x\}_{x\in X}$ in the definition
  of expanding transformation. Also there is $e_0>0$ such that for every $x\in X$
  there is some $U_y$ such that the ball $B(x,e_0)\subset U_y$.

 We have that $e_0>0$ and $0<\la<1$ are such that for every $x\in X$ 
 the branches of the inverse of $T$ are well defined, injective, 
 have disjoint images 
 and are $\la$-contractions on the ball $B(x,e_0)$ of radius $e_0$ centered at $x$.
 
Given $F\in\Lip(X,\re)$, the Lax operator for $F$ is 
  $\cL_F:\Lip(X,\re)\hookleftarrow$
  $$
  \cL_F(u)(x)=\max_{y\in \t^{-1}(x)} \big\{\a+F(y)+u(y)\big\},
  $$
  where
  $$
  \a =\a(F):= -\max_{\mu\in\cM(\t)}\int F\, d\mu.
  $$

  Denote the set of maximizing measures by 
  $$
  \cM(F):=\Big\{\, \mu\in \cM(T)\;\Big|\; \int F \;d\mu = -\a(F)\;\Big\}.
  $$
  
  A {\it calibrated sub-action} for $F$ is a fixed point of the Lax operator $\cL_F$.
  
  \pagebreak
  
  \begin{Lemma}\label{LoF}
  \begin{enumerate}\quad
  \item[1.]  If $u\in\Lip(X,\re)$, the Lipschitz constants satisfy
  \begin{equation}\label{LFu}
  \Lip(\cL_F(u)) \le \la \big(\Lip(F)+\Lip(u)\big).
  \end{equation}
  In particular 
  $\cL_F(\Lip(X,\re))\subset \Lip(X,\re)$.
  \item[2.]   If $\cL_F(u)=u$, writing
    \begin{equation}\label{ovF}
  \oF:=F+\a(F)+u-u\circ \t
  \end{equation}
  we have that
    \begin{enumerate}
  \renewcommand\theenumii{\roman{enumii}}
  \item\label{1i}  $\displaystyle \a(\oF)=-\max_{\mu\in\cM(\t)}\int \oF\,d\mu=0$.
  \item\label{1ii}  $\oF\le 0$.
  \item\label{1iii} $\cM(F)=\cM(\oF)=\{\,\t\text{-invariant measures supported on }[\oF=0]\,\}$
  \end{enumerate}
  
  \item[3.] If $u\in\Lip(X,\re)$ and $\be\in \re$ satisfy $\cL_F(u)=u+\be$,
  then $\be =0$.
  \end{enumerate}
  \end{Lemma}
  
  \begin{proof} \quad
  \begin{enumerate}
  \item[1.] Given $x, y\in X$ with $d(x,y)<e_0$, let $\oy\in T^{-1}(y)$ be such that
  $$
  \cL_F(u)(y) = \a+F(\oy)+u(\oy).
  $$
  Let $S:B(y,e_0)\to X$ be the branch of the inverse of $T$ such that 
  $S(y)=\oy$. We have that
  \begin{align*}
  \cL_F(u)(y)-\cL_F(u)(x)
   &\le \a + F(\oy)+u(\oy) -\a -F(S(x))-u(S(x))
   \\
   &\le F(S(y))-F(S(x))+u(S(y))-u(S(x))
   \\
   &\le \la\,\big(\Lip(F)+\Lip(u)\big)\, d(y,x).
  \end{align*}
  The other inequality is similar.
 
  \item[2.] Observe that for any invariant probability $\mu$ we have that
  \begin{equation}\label{oF}
  \int \oF d\mu = \a + \int F\,d\mu.
  \end{equation}
  Therefore
  $$
  -\a(\oF) = \max_{\mu\in\cM(T)}\int\oF\,d\mu
  =\a(F)+ \max_{\mu\in\cM(T)}\int F\,d\mu
  = \a(F)-\a(F) = 0.
  $$
  This gives (i).
  
   (ii). Since $\cL_F(u)=u$, we have that
   \begin{align*}
   u(T(y)) \ge \a + F(y) + u(y) \qquad \forall y\in X.
   \end{align*} 
   Thus $\oF\le 0$.
   
   (iii).  By the equality \eqref{oF} we have that $\cM(F)=\cM(\oF)$.
   Since $T$ is continuous,
   under the weak* topology, the space $\cM(T)$ of invariant measures 
   is closed in the space of Borel probabilities in $X$, which is compact.
   Since $\oF$ is continuous, the map $\mu\mapsto \int\oF\, d\mu$ is
   continuous. Therefore the maximum in (i) is attained by an invariant
   probability.
   
   By (ii) the function $\oF\le 0$ is non-positive. Therefore any invariant
   measure supported on $[\oF=0]$ is maximizing for $\oF$.
   Conversely, by (i), if $\mu$ is a maximizing measure for $\oF$ then it is invariant and
    $\int\oF\,d\mu=0$.
   Thus the support of $\mu$ must be inside $[\oF=0]$.
   
   \item[3.]  Define $\oF$ by \eqref{ovF}. The hypothesis $\cL_F(u)=u+\be$ implies
   that $\oF(y)\le \be$ for all $y\in X$. Therefore
   $$
   \be \ge \max_{\mu\in\cM(T)}\int\oF\,d\mu
   = \a + \max_{\mu\in\cM(T)}\int F\, d\mu =0.
   $$

   The set $[\oF=\be]$ is closed
   and by the hypothesis $\cL_F(u)=u+\be$, it contains a whole pre-orbit.
   This means that there is a sequence $\{x_n\}_{n\in\na}\subset [\oF=\be]$
   such that $\forall n\in\na$, $T(x_{n+1})=x_n$.
   Let $\mu_N$ be the probability measure defined by
   $$
   \int f \; d\mu_N := \frac 1N \sum_{i=0}^{N-1} f(x_i),
   \qquad \forall f\in C^0(X,\re).
   $$
   Since $X$ is compact, the space of Borel probability measures on $X$ is compact.
   Therefore there is a convergent subsequence $\lim_k \mu_{N_k}=\nu$. 
   The probability $\nu$ is supported on $[\oF=\be]$ and it is  $T$-invariant.
   We have that
   $$
    \be=\int \oF\,d\nu =\a +\int F\, d\nu \le 0.
   $$

  \end{enumerate}
  \end{proof}

  For $f:X\to \re$ continuous, write
  $$
  \lV f\rV_0:=\sup_{x\in X}|f(x)|.
  $$
 
  \medskip
  
  \begin{Proposition}\label{wK}
  There exists a Lipschitz calibrated sub-action.
  \end{Proposition}
  
  \begin{proof}
  By \eqref{LFu}, the Lax operator $\cL_F$ leaves invariant the space
  $$
  \E :=\left \{\, u\in\Lip(X,\re)\;\Big|\; \Lip(u)\le \frac{\la \,\Lip(F)}{1-\la}\right\}.
    $$
      
     Fix $x_0\in X$. 
     Arzel\`a-Ascoli Theorem implies that 
    the quotient space $\E/\re:=\E/\{\text{constants}\}$ with the 
    supremum norm $\lV f+\re\rV_{\E/\re}:=\sup_{x\in X}|f(x)-f(x_0)|\le 2\lV f \rV_0$  
    is compact. 
    
    If $a\in\re$ then $\cL_F(u+a)= \cL_F(u)+a$. Therefore 
    $\cL_F:\E/\re\to \E/\re$ is well defined.
    If $u,\,v\in \E$,  $x\in X$ and 
    $x^*_u\in X$ is such that
    $\cL(u) = \a+F(x^*_u)+u(x^*_u)$, then
    \begin{align*}
    \cL_F(u)(x)-\cL_F(v)(x) &\le \a +F(x^*_u)+u(x^*_u)-\a-F(x^*_u)-v(x^*_u)\\
    &\le u(x^*_u)-v(x^*_u) \le \lV u-v\rV_0.
    \\
    \lV\cL_F(u)-\cL_F(v)\rV_0 &\le \lV u-v\rV_0.
    \\
    \Vert(\cL_F(u)+\re)-(\cL_F(v)+\re)\Vert_{\E/\re} &\le 2  \lV u-v\rV_0.
    \end{align*}
    Choosing representatives for $u+\re$ and $v+\re$ such that $u(x_0)=v(x_0)$,
    we have that $\lV u-v\rV_0=\lV (u+\re)-(v+\re)\rV_\E$. Thus
    $$
    \lV\cL_F(u+\re)-\cL_F(v+\re)\rV_\E\le 2 \lV(u+\re)-(v+\re)\rV_\E.
    $$

    Therefore the space $\E/\re$
    is compact and convex and on it $\cL_F$ is continuous.
    By Schauder Theorem \cite[Theorem 18.10, p.~197]{GK} $\cL_F$ has a 
    fixed point in $\E/\re$.
    In fact $\cL_F$ is non-expanding in   the supremum norm and a simpler fixed point
    applies\footnote{
                     Let $\F=\E/\re$ with the norm $\lV u+\re\rV_\F
                     :=\min_{a\in\re}\lV u+a\rV_0$. Then $(\F,\lV\cdot\rV_\F)$ 
                     is compact, convex
                     and $\cL_F$ has Lipschitz constant 1 on $\lV\cdot\rV_\F$.}
     \cite[Theorem 3.1, p.~28]{GK}. 
  
    Then there is $u\in \E$ and $\be\in \re$ such that $\cL_F(u)=u+\be$.
    By Lemma~\ref{LoF}-3, we have that $\be=0$.
  
  \end{proof}

    \bigskip

    If $u$ is a calibrated sub-action, every point $z\in X$ has a {\it calibrating pre-orbit},
    $(z_k)_{k\le 0}$ such that $T(z_{-k})=z_{-k+1}$, $\t^i(z_{-i})=z_0=z$ and
    \begin{equation}\label{cal}
    u(z_{k+1}) = u(z_{k}) + \a+ F(z_k), \qquad \forall k\le-1.
    \end{equation}
    Or equivalently, since $T(z_k)=z_{k+1}$, 
    \begin{equation}\label{cal2}
    \oF(z_k)= 0, \quad \forall k\le -1.
    \end{equation}

    The iteration of equality \eqref{cal} gives
    \begin{equation}\label{sumcal}
    \forall k\le -1,\quad
    u(z_0) = u(z_{-k})+k\a+\sum_{i=-k}^{-1} F(z_i)
    \end{equation}
    for any calibrating pre-orbit.

   \begin{Lemma}\label{rp}
    
   If there is a periodic orbit $\cO(y)$ such that for any calibrated
   sub-action the $\a$-limit
   of every calibrating pre-orbit is $\cO(y)$ then
   every maximizing measure has support on $\cO(y)$.
   \end{Lemma}

     \begin{proof}
     It is enough to prove the following
     
      \begin{Claim}
      If $\nu$ is an ergodic maximizing measure
     there is a Borel set $Y$ with $\nu(Y)=1$ 
     such that for any $y\in Y$
     there is a calibrating pre-orbit $\{x_n\}_{n\in\na}$ of a
     calibrated sub-action $u$ such that $y\in\a$-lim$\{x_n\}_{n\in\na}$.
     \end{Claim}
     
     We will prove it by applying Poincar\'e Recurrence Theorem to the 
     inverse of the natural (bijective) extension of $T$.

    There is a canonical way of embedding an expanding map into an
    invertible map as follows. Let $\X\subset X^\na=\prod_{n\in\na}X$
    be the space of sequences $\{x_n\}_{n\in\na}$ with $T(x_{n+1})=x_n$
     for every $n\in\na$, endowed with the subspace topology induced by the
     product tolopogy on $X^\na$. Since $X$ is compact, by Tychonof Theorem
     $X^\na$ is compact and then, as a closed subspace, $\X$ is compact.
      Let $\T:\X\to \X$ be defined by 
     $\T(\{x_n\}_{n\in\na})=\{T(x_{n})\}_{n\in \na}=\{\ldots, x_1,x_0,T(x_0)\}$.
     Then $\T$ is a homeomorphism with inverse
     $\T^{-1}(\{x_n\}_{n\in\na})=\{x_{n+1}\}_{n\in\na}$. We have the
     semiconjugacy
     $$
     \begin{CD}
     \X                     @> \T >>   \X  \\
     @V \pi_0 VV                      @VV \pi_0 V  \\
     X                       @> T >>   X
     \end{CD}
     $$
     given by   $\pi_0(\{x_n\}_{n\in\na})=x_0$.
     The projection $\pi_0$ is continuous and hence Borel measurable.
     There is a natural way of lifting invariant measures as follows 
     (cf. Bowen~\cite[\S 1.C]{Bowen0}). If $f\in C^0(\X,\re)$ define $f^*\in C^0(X,\re)$
     by
     $$
     f^*(x) 
     =\min f(\pi_0^{-1}\{x\}).
     $$
     If $\mu$ is a $T$-invariant Borel probability on $X$ 
     define $\tilde{\mu}$ on Borel($\X$) by
     $$
     \tilde{\mu}(f) := \lim_n \mu((f\circ \T^n)^*),
     \qquad \forall f\in C^0(\X,\re). 
     $$
     Then $\tilde{\mu}$ is $\T$-invariant and $(\pi_0)_*(\tilde{\mu})=\mu$.
     
     Suppose that $\nu$ is an ergodic maximizing measure for $F\in\Lip(X,\re)$
     and let $\tnu$ be its invariant lift to $\X$ as defined above.
     The measure $\tnu$ is $\T$-invariant and thus also $\T^{-1}$-invariant. 
     Then $\supp(\tnu)$ is $\T^{-1}$-invariant.
     Let $\Y$ be the 
     set of $\T^{-1}$-recurrent points in $\supp(\tnu)$ and $Y:=\pi_0(\Y)$.
     Then $\nu(Y)=\tnu(\pi_0^{-1}(Y))\ge \tnu(\Y)=1$.
     If $y\in Y$ then there is $\ty\in\pi_0^{-1}(y)\in\Y$ 
     such that $\ty$ is $\T^{-1}$-recurrent, i.e. $\ty\in\om\text{-lim}(\ty,\T^{-1})$.
     We have that $\ty=\{y_n\}_{n\in\na}$ is a pre-orbit of $T$ in $\supp(\nu)$
     with $y_0=y$ and $y\in \a\text{-lim}(\{y_n\}_{n\in\na})$.
     
     Let $u$ be any calibrated sub-action. Let $\oF$ be defined by \eqref{ovF}.
     By Lemma~\ref{LoF}-2.(iii) we have that 
     $\{y_n\}_{n\in\na}\subset\supp(\nu)\subset [\oF=0]$.
     Thus by the remark in~\eqref{cal2} the pre-orbit $\{y_n\}_{n\in\na}$ calibrates $u$.
      \end{proof}

%

    \vskip .7cm

   We say that a sequence $(x_n)_{n\in\na}\subset X$ is a {\it $\de$-pseudo-orbit}
   if $d(x_{n+1},\t(x_n)) \le \de$, $\forall n\in \na$.

  We say that the orbit of $y$ {\it $\e$-shadows} a pseudo-orbit  
  $(x_n)_{n\in\na}$ if $\forall n\in\na$, $d(\t^n(y),x_n)<\e$.

   \begin{Proposition}[Shadowing Lemma]\label{shadowing}\quad
   
   If $(x_k)_{k\in\na}$ is a $\de$-pseudo-orbit  with $\de<(1-\la) e_0$ 
   then there is $y\in X$ whose 
   orbit 
   \linebreak
   $\e$-shadows $(x_k)_{k\in\na}$ with $\e=\frac \de{1-\la}$.
   If $(x_k)_{k\in\na}$ is a periodic pseudo-orbit then $y$ is a periodic orbit with the same period.
   
   \end{Proposition}
   
   \begin{proof}
   Write $B(x,r):=\{\,z\in X\,|\, d(z,x)\le r\,\}$ and $a:=\frac{\la\, \de}{1-\la}$.
   Let $\ti_k$ be the branch of the inverse of $T$ such that 
   $\ti_k(T(x_k))=x_k$. Since $a+\de< e_0$, we have that 
   \begin{align*}
   \ti_k\big(B(x_{k+1},a)\big)\subseteq \ti_k\big(B(\t(x_k),a+\de)\big)
   \subseteq B\big(x_k,\la(a+\de)\big)=B(x_k,a).
   \end{align*}
   Let $y\in X$ be given by
   $$
   y\in \bigcap_{k=0}^\infty \ti_0\circ\cdots\circ\ti_k\big(B(x_{k+1},a)\big).
   $$
   The point $y$ exists and is unique because it is the intersection of a nested family of 
   non-empty compact sets with diameter smaller than $2a\la^k $. 
   We have that  $\t^k(y)\in B(x_k,a)$. Thus $y$ $a$-shadows $(x_k)$.
   Now suppose $(x_k)$ is $p$-periodic. Then also $\t^p(y)$  $a$-shadows $(x_k)$.
   The uniqueness of $y$ implies that $T^p(y)=y$.
    
   \end{proof}
   
   \begin{Corollary}\label{expansivity}\quad
   
   If $T^p(y)=y$ and $(z_k)_{k\le 0}$ is a pre-orbit which $(1-\la)e_0$-shadows the orbit $\cO(y)$
   of $y$,  i.e. $\forall k\le 0$,  $T(z_k)=z_{k+1}$ and 
   $d(z_k,T^{k\!\!\! \mod\! p}(y))<(1-\la) e_0$,
   then the $\a$-limit of $(z_k)$ is $\cO(y)$. 
   \end{Corollary}
   
   \begin{proof}
   Let $w\in\a\text{-lim}(z_k)$. Then there is a sequence $k_n\to -\infty$ such that
   $\lim_n z_{k_n}=w$. Extracting a subsequence if necessary, we may assume that
   $k_n\,(\text{mod}\,p)$ is constant. Then there is $\ell\in \Z_p$ such that 
   $d(z_{k_n},T^\ell(y))<(1-\la) e_0$ for all $n$. The argument in Proposition~\ref{shadowing}
   shows that $d(z_{k_n}, T^\ell(y))< \la^{k_n}\,e_0$. Therefore 
   $w=\lim_n z_{k_n}=T^\ell(y)\in\cO(y)$. It follows that $\a\text{-lim}(z_k)=\cO(y)$.
   
   ~\end{proof}
   
   We show now a condition which allows to obtain a perturbation with 
    maximizing measure supported on a periodic orbit. The argument appeared first in 
    Yuan and Hunt \cite{YH}. The proof below is a modification that we shall need
    of the arguments by Quas and Siefken \cite{QS} which we adapt to  pseudo-orbits.
  
  Let $y\in\Per(T)=\cup_{p\in\na^+}\Fix(T^p)$ be a periodic point for $T$. 
  Let $P_y$ be the set of Lipschitz functions 
  $F\in\Lip(X,\re)$ such that there is a unique $F$-maximizing measure and
  \noindent it is supported
  on the positive orbit of $y$. Let $\cU_y$
  be the interior of $P_y$ in
  $\Lip(X,\re)$.

  \begin{Proposition}\label{perturb}
  Let $F,u\in\Lip(X,\re)$ with $\cL_F(u)=u$ and let $\oF$ be defined by \eqref{ovF}.
  
  Suppose that there exists $M\in\na^+$ such that for every $Q>1$ and $\de_0>0$ 
  there exist $0<\de<\de_0$ and a $p(\de)$-periodic $\de$-pseudo-orbit $(x_k^\de)_k$ in $[\oF=0]$
   with at most $M$ jumps
  such that $\frac{\ga_\de}{\de}\ge Q$, where
  $\ga_\de:= \min_{0\le i<j< p(\de)} d(x_i^\de,x_j^\de)$.
  
  Then $F$ is in the closure of $\cup_{\text{$y$\,periodic}\;} \cU_y$.
  \end{Proposition}

  \begin{proof}
  Observe that fixing $u\in\Lip(X,\re)$, for any $H\in\Lip(X,\re)$ the functions $H$ 
  and $H+\a(F)+u-u\circ T$
  have the same maximizing measures. Therefore it is enough to prove that the function $\oF$
  is in the closure of $\cup_{\text{$y$\,periodic}\;} \cU_y$.
  
  Let $\e>0$. We will show a perturbation  of $F$ with Lipschitz norm smaller than $\e$ such
  that it has a unique maximizing measure supported on a periodic orbit. Moreover, we will
  exhibit a neighborhood of the perturbed function in which the same periodic orbit is the unique
  maximizing measure for all functions in the neighborhood.
  The neighborhood will depend on the periodic orbit.
  
  Let 
  \begin{align*}
  K&:=\max\left\{\frac{M\Lip(\oF)}{(1-\la)^2},\frac{\Lip(\oF)+2}{1-\la}\right\},\\
  \rho&:=\frac{3 K \de}\e,\\
  \ga_3& := \textstyle\frac 1{\Lip(T)}\left(\ga_\de- \frac{2\de}{1-\la}\right) - \la\rho.
  \end{align*}

  Assume that $\de$, $\ga_\de$  and $\frac \de{\ga_\de}$ 
  are so small that $\de$, $\rho$, $\ga_3$   are all positive,
    smaller than $(1-\la) e_0$
    and that
  \begin{align}
  2 K \de -\e \rho &=:-2b<0.
  \label{defb}
  \\
  2 K \de + K \rho - \e \ga_3 &=: -2a < 0.
  \label{defa}
  \end{align}
  
  Let $y$ be the $p$-periodic point which $\big(\frac\de{1-\la}\big)$-shadows $(x_k)$.
  Write $y_k:=\t^k(y)$ and 
  $$
  \cO(y)=\{ \t^i(y) \;|\; i= 0,\ldots, p-1\} =\{y_0,\ldots,y_{p-1}\}.
  $$
  For a function $G:X\to\re$ write
  $$
  \langle G\rangle (y) = \frac 1p \sum_{i=0}^{p-1} G(T^i(y)).
  $$
  Let $n_i$, $i=1,\ldots,\ell$, $\ell\le M$, be the jumps of $(x_k)$; i.e.
  $d\big(\t (x_k) , x_{k+1}\big) = 0$ if $k\in\{0,\ldots,p-1\}\setminus\{n_1,\ldots,n_\ell\}$.
  Using Proposition~\ref{shadowing}, we have that
  \begin{align*}
\left|\sum_{k=1+n_{i-1}}^{n_i}\oF(y_k)-\sum_{k=1+n_{i-1}}^{n_i}\oF(x_k)\right|
&\le \sum_{k=1+n_{i-1}}^{n_i} \Lip(\oF)\; d(y_k,x_k)
\le \sum_{k=1}^{n_i-n_{i-1}} \la^{k-1}\; \frac\de{1-\la}\;\Lip(\oF)\\
&\le \frac{\Lip(\oF)}{(1-\la)^2}\;\de.
  \end{align*}
  Thus
  $$
  \left|\sum_{k=0}^{p-1} \oF(y_k)-\sum_{k=0}^{p-1} \oF(x_k)\right|
  \le \frac{M\,\Lip(\oF)}{(1-\la)^2}\;\de.
  $$
  By hypothesis $\forall k$, $\oF(x_k)=0$, thus $\sum_0^{p-1} \oF(x_k)=0$. Therefore
  $$
  \sum_{k=0}^{p-1} \oF(y_k)\ge - \frac{M\,\Lip(\oF)}{(1-\la)^2}\;\de \ge -K \de,
  $$ 
  \begin{equation}\label{AF}
  \langle \oF\rangle(y) \ge - \frac{K\de}p.
  \end{equation}
  
  Observe that if $0\le i<j<p$,
  $$
  d(y_i,y_j)\ge-d(y_i,x_i)+d(x_i,x_j)-d(x_j,y_j)\ge\ga_\de-\frac{2\,\de}{1-\la}=:\ga_2.
  $$

  \begin{Claim}Assume that $d(z,y_k)\le \rho\ll e_0$. 
   Take $w_1\in \t^{-1}\{z\}$ such that
  $d(w_1,y_{k-1})<\la \rho$. If $w_2\in \t^{-1}\{z\}\setminus\{w_1\}$ then
  $$
  d(w_2,\cO(y)) \ge \ga_3:=\tfrac{\ga_2}{\Lip(T)}-\la \rho\gg \de.
  $$
 \end{Claim}

  {\it Proof:} Let $y_j\in\cO(y)$ be such that $d(w_2,\cO(y))=d(w_2,y_j)$. 
  
  Let $S$ be the branch of the inverse of $T$ such that $S(z)=w_1$.
  If $x,y \in B(z,e_0)$ then
  $$
  d\big(S(x),S(y)\big) \ge \Lip(T)^{-1} d\big(T(S(x)),T(S(y))\big)
  =\Lip(T)^{-1} d(x,y).
  $$
  This implies that $\la\ge \Lip(T)^{-1}$.
  We also get that $B(w_1,(\Lip T)^{-1}e_0)\subset S(B(z,e_0))$ and then $T$ is 
  injective in the ball $B(w_1, \Lip(T)^{-1} e_0)$. 
  In particular $d(w_2,w_1)\ge \Lip(T)^{-1} e_0$.
  
  If $j={k-1}$ then 
  \begin{align*}
  d(w_2,\cO(y))&=d(w_2,y_j)=d(w_2,y_{k-1})\ge d(w_2,w_1)-d(w_1,y_{k-1})
  \\
  &\ge \Lip(T)^{-1} e_0- \la\,\rho > \ga_3.
  \end{align*}

  If $j\ne k-1$ then
  $$
  \ga_2\le d(y_k,y_{j+1})\le d(y_k,z)+d(z,y_{j+1})\le \rho+\Lip(\t)\,d(w_2,y_j).
  $$ 
  $$
  d(w_2,y_j) \ge \tfrac{\ga_2}{\Lip(T)} - \tfrac{\rho}{\Lip(T)}
  \ge \tfrac{\ga_2}{\Lip(T)} - \la \rho.
  $$
  This proves the claim.

  Now we make two perturbations to $\oF$. The first perturbation 
  is the addition of $-\e g(x)$, where 
  $$
  g(x):=d(x,\cO(y)).
  $$
  This is a perturbation with 
  $$
  \lV \e g \rV_0\le \e \diam X,
   \qquad
   \Lip(\e g) = \e.
   $$
  The second is a perturbation by any function with
  \begin{equation}\label{h0}
  \lV h\rV_0 < \frac{ K \de}{2p},
  \qquad  \Lip(h)\le 1.
  \end{equation}
  This perturbation depends on $\cO(y)$, and in particular on its period $p$. 
  We shall prove that the function $G_1:= \oF-\e g+ h $
  has a unique maximizing measure supported on the periodic orbit $\cO(y)$.
  Since the set of such functions $G_1$  contains an open ball centered at $\oF-\e g$,
  this proves the proposition.
  
  Let 
  \begin{equation}\label{G}
  G = \oF -\e g + h + \be = G_1+\be,
  \end{equation}
  where 
  \begin{equation}\label{defbeta}
  \be = -\sup_{\mu\in\cM(\t)}\int (\oF-\e g +h) \;d\mu.
  \end{equation}
  It is enough to prove the claim for $G$ because 
  $G$ and $G_1$ have the same maximizing measures.
  
  Using \eqref{AF}, we have that  
  \begin{align}
  \be &\le -\langle \oF-\e g+h\rangle (y) 
  =- \langle \oF+h\rangle (y) 
  \notag\\
  &\le - \langle \oF\rangle (y)+\lV h\rV_0
  \notag\\
  &\le \frac{K\de}p +\lV h\rV_0
  \label{beta}
  \end{align}
  
  Let $v$ be a calibrated sub-action for $G$, $\cL_G(v)=v$.
  Given any $z\in X$ let $(z_k)_{k\le 0}$ be a pre-orbit of $z$ which calibrates $v$.
  Let $0>t_1>t_2>\cdots$ be the times on which $d(z_k,\cO(y))>\rho$.
  If $t_{n+1}<t_n-1$ there is $s_{n}\in\Z$ such that 
  the orbit segment $(z_k)_{k=t_{n+1}+1}^{t_n-1}$ $\rho$-shadows 
  $(y_{-i+s_{n}})_{i=t_n-t_{n+1}-1}^1$, thus
  $$
  d(z_{-i+t_{n}},y_{-i+s_{n}})\le \la^{i-1}\, \rho,
  \quad \forall n\in\na, \quad \forall i=1,\ldots, t_n-t_{n+1}-1.
  $$
  By the Claim, we have that
  \begin{equation}\label{tg3}
  t_{n+1}<t_n-1 	\qquad\then\qquad
  d(z_{t_{n+1}},\cO(y))\ge \ga_3.
  \end{equation}
  
  Since both terms in $\oF-\e g$ are non-positive, 
  from \eqref{G} and \eqref{beta} we obtain
  \begin{equation}\label{hb}
  G\le h+\beta \le  \frac{ K \de}p + 2 \lV h\rV_0.
  \end{equation}
  
  On a shadowing segment we have
  \begin{equation}\label{shadowbound}
  \lv \sum_{t_{n+1}+1}^{t_n-1} G(z_k) -\sum_{s_n-t_n+t_{n+1}+1}^{s_n-1}G(y_k)\rv
  \le \Lip(G) \sum_{i=0}^{+\infty} \la^i \, \rho
  \le \Lip(G) \;\frac\rho{1-\la}
  \le K \rho.
  \end{equation}
  Write
  $$
  t_n-t_{n+1}-1 = m p + r
  $$
  with $0\le r<p$ and separate the shadowing segment in $m$ loops along the orbit $\cO(y)$ and a
  residue with at most $p-1$ iterates. Using \eqref{hb} for $(p-1)$ times 
  and~\eqref{shadowbound}, we have that
  $$
  \sum_{t_{n+1}+1}^{t_n-1}G(z_k) 
  \le m p\; \langle G\rangle(y)+(p-1) \frac{K \de}p + 2 (p-1) \lV h \rV_0 + \Lip(G)\, \frac{\rho}{1-\la}.
  $$
  By the definition of $\be$ we have that $\langle G\rangle(y) \le 0$.
  Therefore
   \begin{equation}\label{SG}
  \sum_{t_{n+1}+1}^{t_n-1}G(z_k)
  \le(p-1) \frac{K \de}p + 2 (p-1) \lV h \rV_0 + K \rho.
  \end{equation}
 
  On the points $z_{t_n}$ we have that $d(z_{t_n},\cO(y))>\rho$. 
  Using \eqref{G}, \eqref{hb}, \eqref{h0} and \eqref{defb},
    \begin{align}\label{Gzt1}
  G(z_{t_n})\le \oF(z_{t_n}) -\e\,\rho +\lV h+\be\rV_0
  \le 0 -\e\,\rho +\frac{K\de}p +2\lV h \rV_0
  <-b<0.
  \end{align}
  In particular, this holds when $t_n=t_{n-1}-1$.
  
  When $t_{n+1}<t_{n}-1$,
  using \eqref{G}, \eqref{tg3} and \eqref{hb}, we have that
  \begin{equation}\label{Gzt}
  G(z_{t_{n+1}})\le 0 - \e\,\ga_3+\lV h+\be\rV_0\le -\e\,\ga_3 + \frac{K\de}p + 2 \lV h\rV_0.
  \end{equation}
  Thus, adding \eqref{SG} and \eqref{Gzt}, and using \eqref{h0} and \eqref{defa},
  \begin{equation}\label{Gtn}
  t_{n+1}<t_n-1
  \qquad\then\qquad
  \sum_{t_{n+1}}^{t_n-1}G(z_k)\le 2p\, \lV h\rV_0 + K \de +K \rho- \e \ga_3 < -a<0.
 \end{equation}

  From \eqref{defbeta} we have that  $\a(G)=0$. 
  Since by definition $(z_k)_{k\le 0}$ is a calibrating pre-orbit for $v$,
   as in \eqref{sumcal}, we have that 
  for all $k<0$, 
  \begin{equation}\label{calG}
  v(z) = v(z_k) + \sum_{i=k+1}^{-1} G(z_i).
  \end{equation}
   Since $v$ is finite, we get that
   $$
   \sum_{-\infty}^{-1}G(z_k) \ge - 2 \lV v\rV_0 > -\infty.
   $$
   From \eqref{Gzt1} and \eqref{Gtn} we obtain that the sequence $t_n$ is finite. 
   Since $\rho<(1-\la) e_0$, from Corollary~\ref{expansivity} 
   we get that
    every calibrating pre-orbit has $\a$-limit $\cO(y)$.
   By Lemma~\ref{rp},
   this implies that every maximizing measure for $G$ 
   has support on $\cO(y)$.

  \end{proof}

  \section{Proof of Theorem~\ref{ThmA}}\label{S2}
  
  \noindent{\bf Proof of theorem~\ref{ThmA}:}
  
  We prove that $\cO:=\bigcup_{y\in\Per(T)}\cU_y$ is open and dense.
  It is clearly open.
  
  Suppose, by contradiction,  that there is a non-empty open set 
  \begin{equation}\label{cU}
  \cW\subset \Lip(X,\re)
  \end{equation} which is 
  disjoint from  $\bigcup_{y\in\Per(T)}\cU_y$. 
  By Theorem~\ref{Tmorris} and Remark~\ref{mergodic}
  we can choose $F\in\cW$ such that 
  it has an ergodic maximizing measure $\mu$ with  entropy 
  \begin{equation}\label{0ent}
  h_\mu(T) = 0.
  \end{equation}
  By Lemma~\ref{LoF}-2.(iii) for any calibrating subaction $u$ for $F$, we have
    that $\supp(\mu)\subset [\oF=0]$,
    where $\ov{F}$ is from \eqref{ovF}.
  Let $q\in\supp(\mu)\subset[\oF=0]$ be a generic point for $\mu$, i.e.
  for any continuous function $f:X\to \re$,
  $$
  \int f \;d\mu = \langle f \rangle(q) = \lim_N \frac 1N \sum_{i=0}^{N-1} f(\t^i(q)).
  $$

  Since $F$ is not in the closure of $\bigcup_{y\in\Per(T)}\cU_y$,
  by Proposition~\ref{perturb} with $M=2$,
  we have the following 
  \begin{Statement}\quad\label{STM}
  
   There is $Q>1$ and $\de_0>0$ such that if 
  $0<\de<\de_0$ and
  $(x_k)_{k\ge0}\subset \cO(q)$
  is a $p$-periodic 
   $\de$-pseudo-orbit with at most 2 jumps
  made with elements of the positive orbit of $q$ 
  then $\ga= \min_{1\le i<j<p}d(x_i,x_j) <\frac 12 Q\de$.
 \end{Statement}

   Let $N_0$ be such that 
   \begin{equation}\label{N0}
   2\,Q^{-N_0}<\de_0.
   \end{equation}
   Fix a point $w\in \supp(\mu)$ for which Brin-Katok Theorem holds \cite{BK}, i.e.
    \begin{equation}\label{wBK}
    h_\mu(T) = -\lim_{L\to+\infty} \frac 1L \,\log\mu\big(V(w,L,\e)\big), 
    \end{equation}
    where $V(w,L,\e)$ is the {\it dynamic ball}:
    \begin{equation}\label{dynball}
   V(w,L,\e) := \big\{\,x\in X\;\big|\; d(T^kx,T^k w)<\e\,,\; \forall k=0,\ldots,L\,\big\}.
   \end{equation}
   
      Given $N>N_0$ let $0\le t_1^N<t_2^N<\cdots$ be all the $\frac 12 Q^{-N}$ returns to $w$, i.e.
  \begin{equation}\label{Nreturns}
  \{t_1^N,t_2^N,\ldots\} = \{ n\in\na\,|\,d(T^nq,w)\le   \tfrac 12 Q^{-N}\}.
  \end{equation}

  \bigskip
  \bigskip
  
  We need the following
  
  \begin{Proposition}\label{expo}
  For any $\ell\ge 0$, \quad
  $
  t_{\ell+1}^N-t_\ell^N \ge \sqrt{2}^{N-N_0-1}.
  $
  \end{Proposition}
 
  Using Proposition~\ref{expo} we continue the proof of Theorem~{A}.
 
    Write
    $$
    B(w,r):= \{\,x\in X\;|\; d(x,w) \le r\,\}.
    $$
    Given $N\gg N_0$, let $f_N:X\to\re$ be a continuous function such that
    $0\le f\le 1$, $f|_{B(w,\frac 12 Q^{-N-1})}\equiv 1$ and 
    $\supp(f)\subseteq B(w,\frac 12 Q^{-N})$.
    Using that $q$ is a generic point for $\mu$ and Proposition~\ref{expo}, we have that 
    \begin{align}
    \mu\big(B(w,\tfrac 12 Q^{-N-1})\big) & \le \int f_N\;d\mu
    = \lim_{L\to+\infty}\frac 1L \sum_{i=0}^{L-1}f_N(T^iq) 
    \notag\\
    &\le  \lim_{L\to+\infty}\frac 1L \;\#\Big\{\, 0\le i<L \;\Big|\; d(T^iq,w)\le \tfrac 12 Q^{-N}\,\Big\}
    \notag\\
    &\le  \lim_{L\to+\infty}\frac 1L \; \#\Big\{\,\ell\;\Big|\;t^N_\ell\le L\;\Big\}
    \notag\\
    &\le \sqrt{2}^{-N+N_0+1}.
    \label{muB}
    \end{align}
   
   Recall that  the dynamic ball about $w$ is
   $$
   V(w,L,\e) := \big\{\,x\in X\;\big|\; d(T^kx,T^k w)<\e\,,\; \forall k=0,\ldots,L\,\big\}.
   $$
   We have that 
   $$
   V(w,L,\e) = S_1\circ\cdots\circ S_L\big(B(T^Lw,\e)\big),
   $$
   where $S_k$ is the branch of the inverse of $T$ such  that
   $S_k(T^kw)=T^{k-1}w$. Therefore
   $$
   V(w,L,\e)\subseteq B(w,\la^L\e).
   $$
   Let $N$ be such that
   $$
   \tfrac 12 Q^{-N-2}\le\la^L\e\le\tfrac 12  Q^{-N-1}.
   $$
   Then
   $$
   -N\le L\,\frac{\log \la}{\log Q} + \frac{\log(2\e)}{\log Q} +2.
   $$
   Using \eqref{muB}, we have that
   \begin{align*}
   \mu\big(V(w,L,\e)\big) \le \mu\big(B(w,\la^L\e)\big)
   \le \mu\big(B(w,\tfrac 12Q^{-N-1})\big) 
   \le \sqrt{2}^{-N+N_0+1}.
  \end{align*}
  \begin{align*}
  \frac 1L \, \log \mu\big(V(w,L,\e)\big)
  &\le \frac 1L\,\big(\log\sqrt{2} \big)\big(-N+N_0+1)
  \\
  &\le \frac{\log \la}{\log Q}\,\log\sqrt{2}
  +\frac 1L\,\big(\log\sqrt{2} \big)\left(2+\frac{\log(2\e)}{\log Q}+N_0+1\right).
  \end{align*}
  
  By Brin-Katok Theorem \cite{BK} and the choice of $w$ in \eqref{wBK},
   we have that
  \begin{equation*}
  h_\mu(T) = 
  -\lim_{L\to+\infty}\frac 1L\,\log\mu\big(V(w,L,\e)\big)
  \ge \frac{\log\la^{-1}}{\log Q}\,\log\sqrt{2} >0.
  \end{equation*}
  This contradicts the choice of $F$ and $\mu$ in \eqref{0ent}.
  Therefore such non-empty open set $\cW$ in \eqref{cU} does not exist.
  This implies that the (open) set $\cO=\bigcup_{y\in\Per(T)}\cU_y$
  is dense.
 \qed

  \bigskip
 \bigskip

  Now we prove

  \addtocounter{Thm}{-1}
  \begin{Proposition}
  For any $\ell\ge 0$, \quad
  $
  t_{\ell+1}^N-t_\ell^N \ge \sqrt{2}^{N-N_0-1}.
  $
  \end{Proposition}

\medskip
 
 \noindent {\bf Proof:}
  For $N\in \na$, let 
  \begin{align*}
  \A_N &:= \{(x,y)\in X\times X\;|\; d(x,y) \le Q^{-N}\}.
 \end{align*}
  From~\eqref{N0} and Statement~\ref{STM}, we get
  
  \begin{Statement}\label{STM2}\quad
  
  If $N>N_0$ and  $(x_k)_{k=0}^{p-1}$
   is a $p$-periodic  $Q^{-N}$ pseudo-orbit 
  in $\cO(q)$ with at most $2$ jumps, then there
  is a $\frac 12 Q^{-N+1}$-return $d(x_i,x_j)<\frac 12 Q^{-N+1}$ with 
  $0\le i<j< p$. 
  In particular $(x_i,x_j)\in \A_{N-1}$
  \end{Statement}

  Write $q_i:=T^i(q)$. From~\eqref{Nreturns}, the sequence 
  $(q_k)_{k=t_\ell^N}^{t^N_{\ell+1}-1}$
  is a periodic $Q^{-N}$ pseudo-orbit in $\cO(q)$ with 1 jump.
  Therefore there is a $Q^{-N+1}$-return $d(q_i,q_j)< \tfrac 12 Q^{-N+1}\le Q^{-N+1}$  with 
  \linebreak
  $t^N_\ell\le i<j< t^N_{\ell+1}$.
   This gives rise to two $Q^{-N+1}$ periodic pseudo-orbits in $\cO(q)$ with at most 2 jumps.
   Namely, $(q_i,\ldots,q_{j-1})$ and $(q_j,\ldots, q_{t^N_{\ell+1}-1},q_{t^N_\ell},\ldots,q_{i-1})$. 
   Each of them imply a  $Q^{-N+2}$ approach... This process will continue as long 
   as $N\ge N_0$.
 
         \begin{figure}[h]
     \resizebox*{14cm}{4cm}{\includegraphics{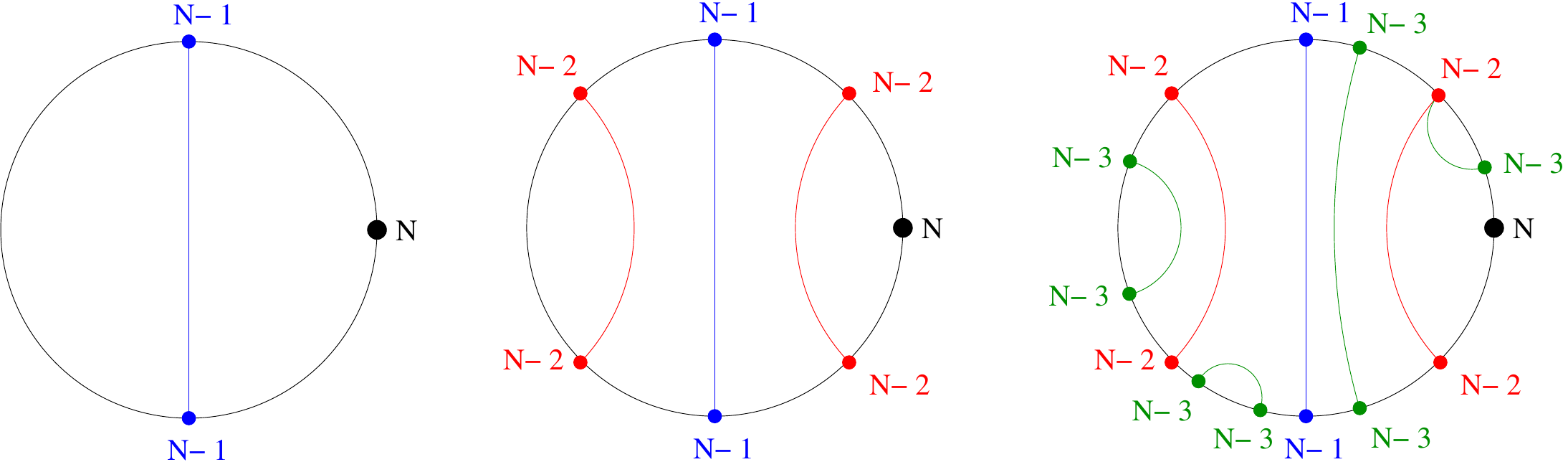}} 
     \caption{The disc $\D$, the circle $\SS=\partial\D$ and an example of a cascade of returns implied 
     by the inductive process.}\label{ex0}
\end{figure}

   It is simpler to show the inductive process in a picture. 
   Draw a circle $\SS$ with the elements of the pseudo-orbit
   $(q_k)_{k=t^N_\ell}^{t^N_{\ell+1}-1}$. Consider a disc $\D$
   with boundary $\partial\D=\SS$.
   Inside the disc $\D$, draw a line from $q_i$ to $q_j$.
   It may be that $q_i=q_{t^N_\ell}$ but in that case $q_j\ne q_{t^N_{\ell+1}}$.
      The line $\ell_1=\ov{q_i q_j}$ separates the disk in two components.
   Each component is a $Q^{-N+1}$ pseudo-orbit with at most two jumps
(one jump of size $\le Q^{-N+1}$ and possibly another with size $\le Q^{-N}<Q^{-N+1}$). 
 Thus, each component has at least one 
   $Q^{-N+2}$ return\,$\ldots$ 
   The interior of the lines in this construction do not intersect.

   We will also draw a tree with the returns, in order to see that their number grows exponentially.
   An example appears in figure~\ref{ex1}.
   The nodes of the tree are the returns implied by Statement~\ref{STM2}.
   The height
    of the node bounds the size of the return. The numbers near a node
   are the quantity of returns in upper levels of the tree which are adjacent to the return of the node, 
   either at its left or at its right. These numbers are also equal to $-1$ +the quantity of jumps
   of the two new periodic pseudo-orbits determined by the node. 
   
      \begin{figure}[h]
     \resizebox*{6cm}{4cm}{\includegraphics{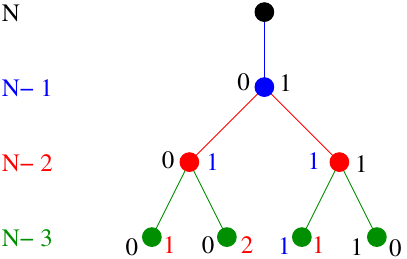}} 
     \hskip 1cm
 \resizebox*{6cm}{5cm}{\includegraphics{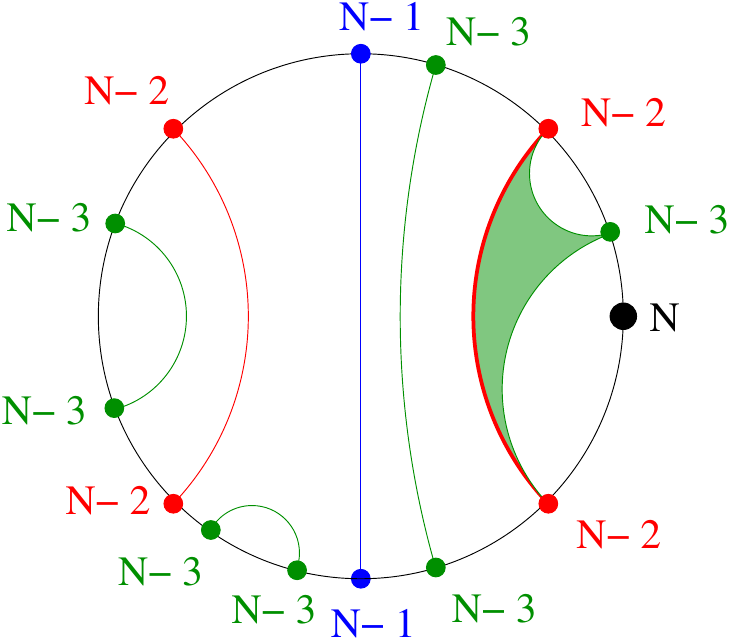}}
  \caption{An example of a distribution of returns implied by Statement~\ref{STM2}
   and the tree representing it. 
   The shadow is explained in \S\ref{c101} and in Figure~\ref{shadow}.}\label{ex1}
\end{figure}
We show how the tree is constructed in the example of figure~\ref{ex1}.
 We begin with a return in $\A_N$. This gives a periodic $Q^{-N}$ pseudo-orbit with no other jump.
 It implies the existence of a return in $\A_{N-1}$. In the tree we draw a vertical line from level $N$ to 
 level $N-1$. At this stage, the line in the circle corresponding to the $\A_{N-1}$ return 
 divides the disk in two components. One side has 1 return in $\A_N$ that appears 
 in a previous level in the tree and the other side has 0 returns appearing above in the tree.
 We write the numbers 0 and 1 at the sides of the node of the tree corresponding 
 to the $\A_{N-1}$ return. The $\A_{N-1}$ return divides the circle in two components.
 The component at the left is a periodic $Q^{-N+1}$ 
 pseudo-orbit with only one $Q^{-N+1}$ jump, 
 corresponding to the number 0 in the tree. The component at the right is a $Q^{-N+1}$ pseudo-orbit 
 with a $Q^{-N-1}$
 jump and also a $Q^{-N}$ jump, and corresponds to the number 1 in the tree in the node
 at level $N-1$.

  Statement~\ref{STM2}
  implies the existence of other returns in $\A_{N-2}$ for both pseudo-orbits.
 In the right hand side of figure~\ref{ex1} we draw the case in which the pseudo-orbit 
 segment between the $\A_{N-2}$ return contains a $Q^{-N}$ jump.
 Cutting the $Q^{-N+1}$ pseudo-orbit of the right hand side of the circle at the $\A_{N-2}$ return
 we obtain two $Q^{-N+2}$ periodic pseudo-orbits. 
 The one at the right has a {\color{black}$Q^{-N}$} 
 jump which appears previously in the tree and the
 one at the left has a {\color{blue}$Q^{-N+1}$} jump appearing previously in the tree. 
 We write the numbers {\color{blue} 1} and {\color{black} 1} in the corresponding 
 node of the tree.
 
    \begin{figure}[h]
     \resizebox*{6cm}{3.5cm}{\includegraphics{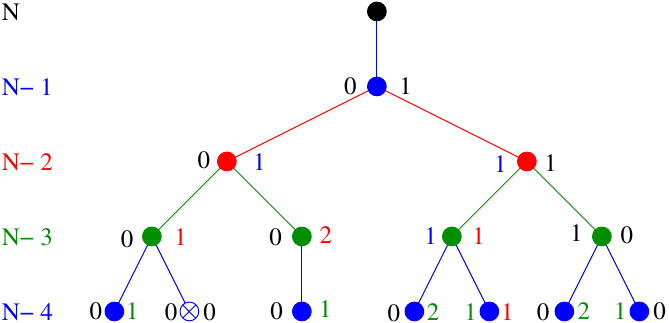}}
      \resizebox*{6cm}{4.9cm}{\includegraphics{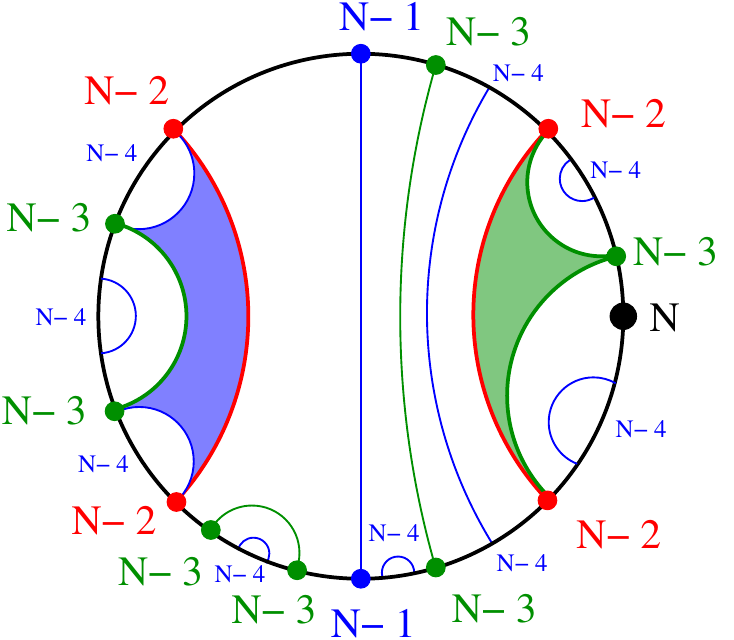}} 
      \caption{\small This is a possible next step from the example in Figure~\ref{ex1}.
      At level ${\color{dgreen}N-3}$ we had a node ${\color{black}0}{\color{dgreen}\bullet}{\color{red} 2}$ 
      which only issues one branch with  label $0$. At level ${\color{blue}N-4}$ the node {\color{blue}$0\otimes 0$ }
      corresponds to the shadowed region on the left of the disc. 
      This node comes from a branch with label $\color{red} 1$, i.e. a periodic specification with $2$ jumps. 
      In this case the implied return in $\color{blue}\A_{N-4}$ has both of its points at  the jumps of
      the specification. We put a white (or empty) node $\color{blue}\otimes$ in the tree, signifying that this 
      node (approach) does not count as a new point in the tree, i.e. as another point in the pseudo-orbit
      $(q_{t^N_\ell},\ldots,q_{t^N_{\ell+1}-1})$ which was not accounted for earlier.
      We show that in this case both jumps of the specification give two approaches
      which issue two periodic specifications with only one jump.
       We write the labels $\color{blue}0$ in the node $\color{blue}0\otimes 0$
      meaning that both implied specifications have only one jump. 
      The node $\color{blue}0\otimes 0$ will issue two branches 
      (with label $0$). We shadow the cuadrilateral region at the left to be not considered later. 
      After drawing the shadow there remain two white regions in the disc
      which give two ${\color{blue} Q^{-N+4}}$ periodic specifications  with only one jump that will restore the duplication process.}
      \label{ex1a}
     \end{figure}

We will  provide the tree with black  nodes $\bullet$ and white (or empty) nodes $\otimes$.
The nodes in the tree are associated to the approaches implied by the process. A black node 
means that at least one of the points in the approach is a point in the pseudo-orbit
$(q_{t^N_\ell},\ldots,q_{t^N_{\ell+1}-1})$ which
did not appear in the previous approaches.
So that we have
 \begin{equation}\label{nodes}
 t^N_{\ell+1}-t^N_\ell \ge \#\{\text{black nodes}\}.
 \end{equation}
 The branches of the tree correspond to the new periodic pseudo-orbits implied 
 by the approach at the node which issues the branches. The numbers at the 
 node are associated to the branches issued by the node.
 The number 2 has no issued branch.
 
 The tree usually duplicates its nodes but we have to be careful of two situations.
 The first is when an approach implies a periodic pseudo-orbit with more than 2 jumps,
 i.e. a number 2 (or more) in the tree. For simplicity we have chosen to limit our accounting 
 to at most 2 jumps.
 In this case Statement~\ref{STM2} does not imply
 the existence of a new approach and we stop the process.
 In the tree this means that there is no new branch corresponding to a number 2.
 We will see that this only happens when the parent node has label $0\bullet 2$
 and that the $0$ side does issue a new branch which restarts the duplication process.

  \begin{figure}[h]
     \resizebox*{12cm}{3.5cm}{\includegraphics{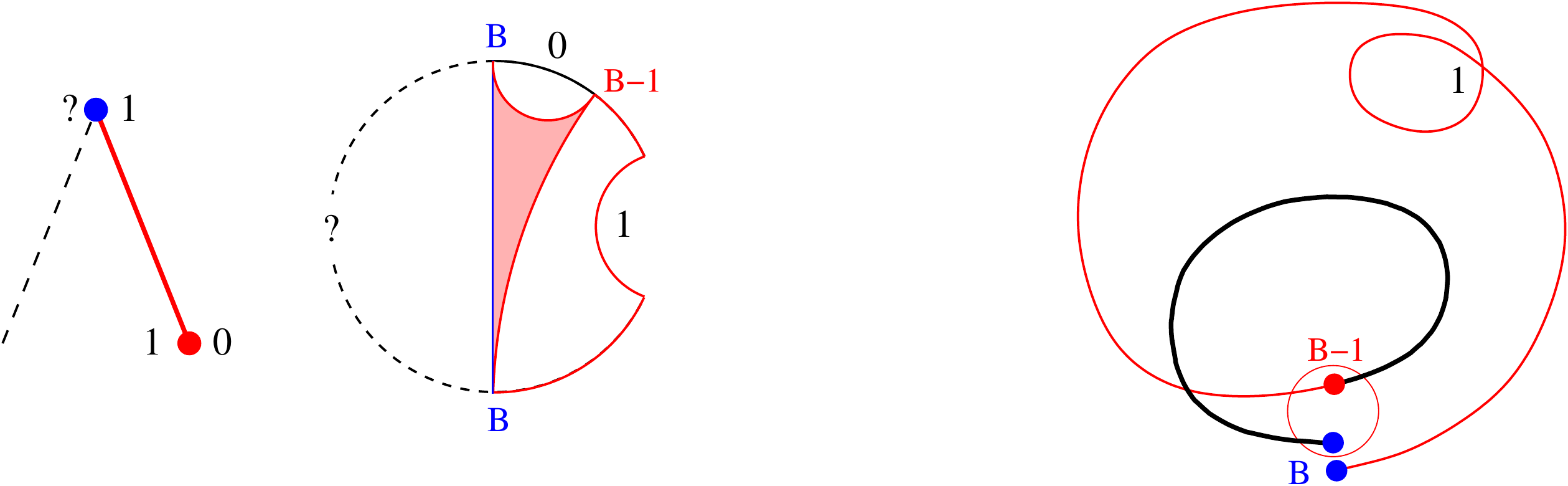}} 
     \caption{If one of the points in the approach  implied by Statement~\ref{STM2} is
     one of the jumps of the mother periodic pseudo-orbit we observe that it
     divides the mother pseudo-orbit in two child pseudo-orbits. 
     We draw lines connecting the ends of these pseudo-orbits and shadow 
     the internal part of the disk $\D$ which does not contain an interval in the circle $\SS$.}\label{shadow}
   \end{figure}

The other situation is when at least one point of a new approach is exactly 
at one of the jumps of the mother pseudo-orbit, see figures \ref{shadow} and \ref{shadow4s}.
We will see that in these cases
the approach implies two new periodic specifications, and hence two new branches
issued from the node corresponding to the approach, which will continue the 
duplication process. In the case when both points in the approach are at the
jumps of previous periodic pseudo-orbit, as in Figure~\ref{shadow4s}, both points may have already been accounted
for previously in the tree. In this case we put a white (or empty) node $\otimes$ in the
tree.

   \begin{figure}[h]
     \resizebox*{12cm}{3.5cm}{\includegraphics{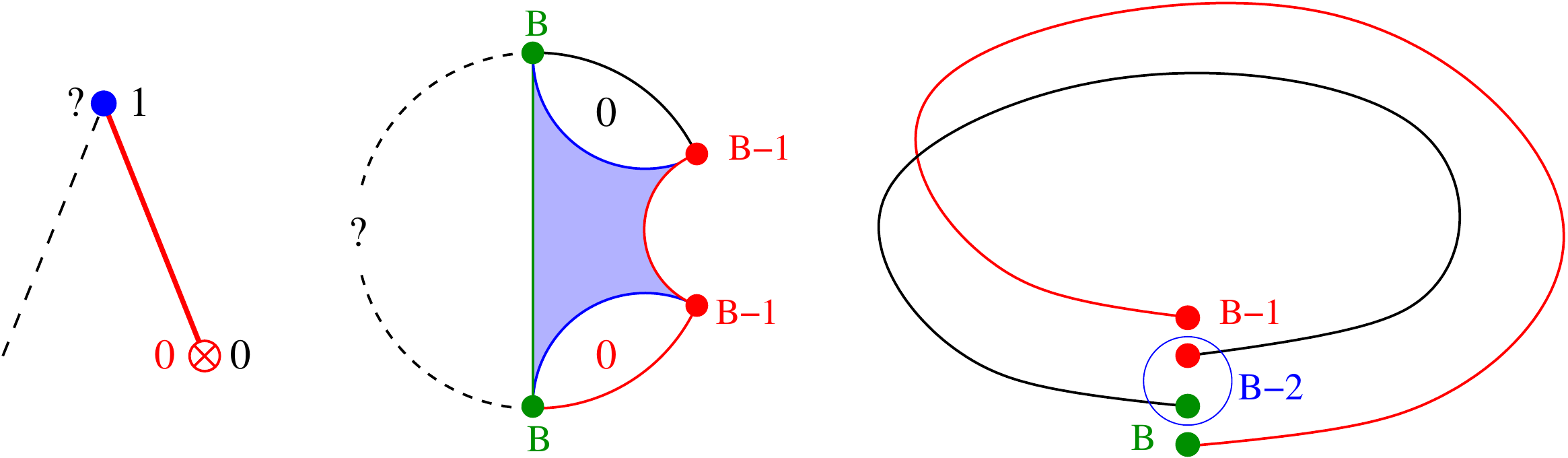}} 
     \caption{\small  If both points of the approach implied by Statement~\ref{STM2} are 
     exactly the jumps  of the mother periodic pseudo-orbit we observe that it
     divides the mother pseudo-orbit in two child pseudo-orbits with only one jump.
     The approach does not necessarily imply a new point in $(q_{t^N_\ell},\ldots,q_{t^N_{\ell+1}})$
     which was not accounted for previously in the tree. 
     Therefore we write a white (or empty) node $\otimes$
     in the tree. 
     We draw lines connecting the ends of these pseudo-orbits and shadow 
     the internal part of the disk $\D$ which does not contain an interval in the circle $\SS$.}\label{shadow4s}
   \end{figure}

We now study the building blocks of the tree.
The case of a periodic $Q^{-B}$ pseudo-orbit with only 1 jump is represented in Figure~\ref{case0},
and the case with 2 jumps is in Figure~\ref{case1s}.

       \begin{figure}[h]
     \resizebox*{12cm}{6cm}{\includegraphics{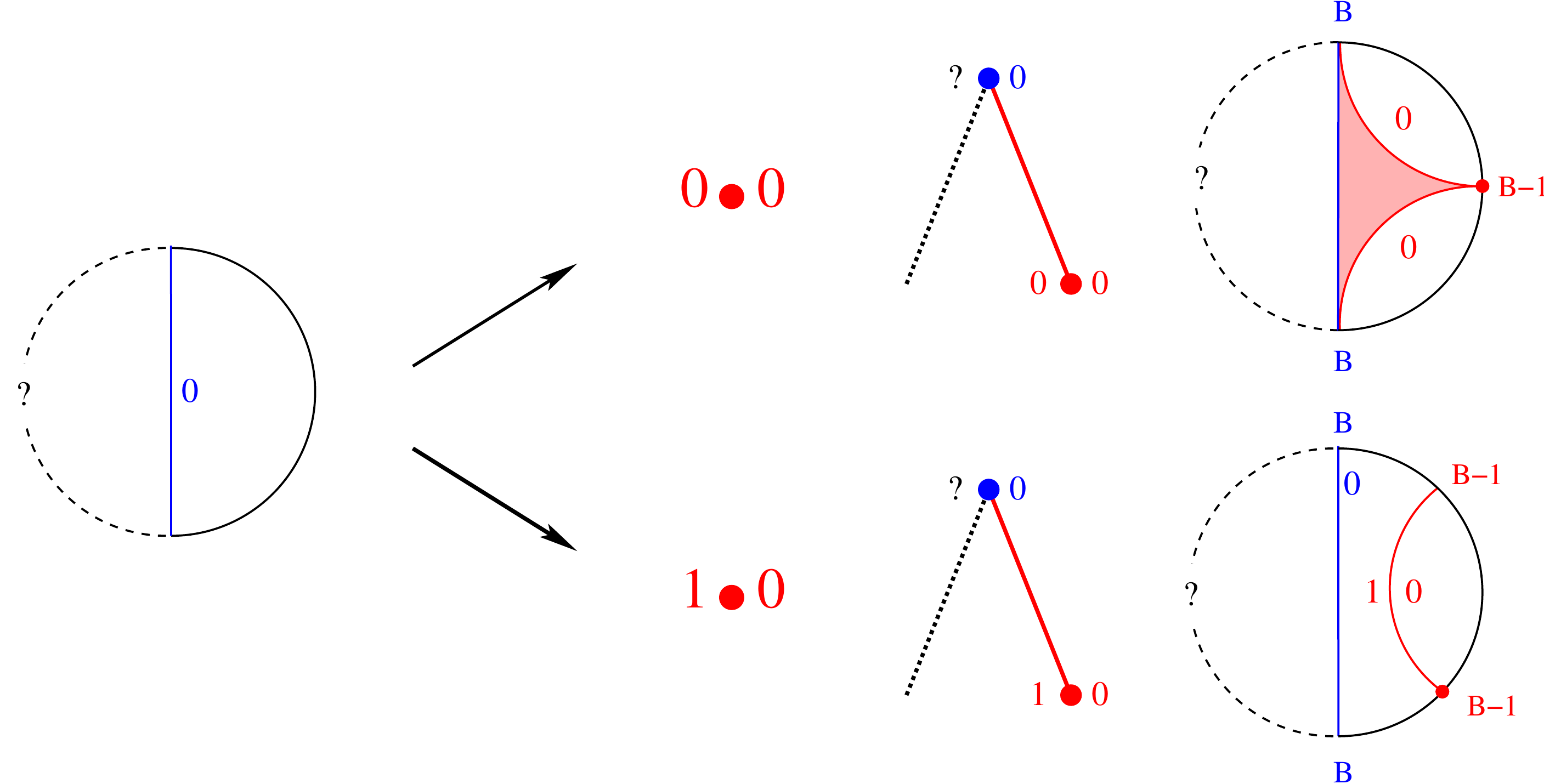}} 
  \caption{\small Possible nodes ending a branch with a label 0, i.e. child specifications 
   of a periodic  1-specification with only one jump.}\label{case0}
\end{figure}

\subsection{Childs of a periodic pseudo-orbit with 1 jump.}\label{scase0}\quad

\subsubsection{Case $0\bullet 0$. When one point of the approach is at the jump of the mother pseudo-orbit} 
Denote the periodic $Q^{-B}$ pseudo-orbit with 1 jump by $(q_a,\ldots,q_{b-1})$, $q_b=q_a$. In this case the 
$\tfrac 12 Q^{-B+1}$ approach is 
$(q_i,q_j)$ with $q_i=q_a$.
Observe that 
\begin{align}
d(T(q_{j-1}),q_i)&=d(q_j,q_i)<\tfrac 12 Q^{-B+1}.
\label{0.0a}\\
d(T(q_{b-1}),q_j)&\le d(T(q_{b-1}),q_a)+d(q_a,q_j)
\notag\\
&=d(T(q_{b-1}),q_a)+d(q_i,q_j)
\le Q^{-B}+\tfrac 12 Q^{-B+1}
<Q^{-B+1}.
\label{0.0b}
\end{align}
From \eqref{0.0a} we have that $(q_i,\ldots,q_{j-1})$ is a $Q^{-B+1}$ pseudo-orbit with only 1 jump and from
\eqref{0.0b} we have that $(q_j,\ldots,q_{b-1})$  is another $Q^{-B+1}$ pseudo-orbit with only 1 jump. 
In the disk $\D$ we draw the lines  $\ov{q_aq_j}$ and $\ov{q_jq_b}$, and also shadow the triangular region limited by the lines 
$\ov{q_aq_b}$, $\ov{q_aq_j}$ and $\ov{q_jq_b}$. This shadowed region is treated as a line with a right and left side.
The choice of right and left sides may be ambiguous and is left to the reader's will. The two regions left in white in the
disk $\D$ correspond to the periodic $Q^{-B+1}$ specifications with only one jump mentioned above.  
In the tree we label the node with the symbol $0\bullet 0$. 
The node is black $\bullet$ because the point $q_j$ in the approach $(q_a,q_j)$ did not 
appear before as a node in the tree. This node will have two branches corresponding to
the numbers 0 and 0.

\subsubsection{Case $1\bullet 0$. When both points of the approach are not at the jump of the pseudo-orbit}
Denote the periodic $Q^{-B}$ pseudo-orbit by $(q_a,\ldots,q_{b-1})$, $q_b=q_a$.
In this case the $\tfrac 12Q^{-B+1}$ approach is $(q_i,q_j)$ with $a<i<j<b$;
it implies two daughter periodic $ Q^{-B+1}$ pseudo-orbits: $(q_a,\ldots, q_{i-1}, q_j,\ldots, q_{b-1})$ with 2 jumps
and $(q_i,\ldots,q_{j-1})$ with only 1 jump. In the tree we label the node as $1\bullet 0$.
The numbers 1 and 0 correspond to the new implied $Q^{-B+1}$ pseudo-orbits with 2 and 1 jumps 
respectively. The node issues two branches corresponding to the numbers 1 and 0. The node is 
black $\bullet$ because the approach $(q_i,q_j)$ has one of its points (in fact both points)
which did not appear before in the nodes of the tree.

      \begin{figure}[h]
     \resizebox*{14.8cm}{7.4cm}{\includegraphics{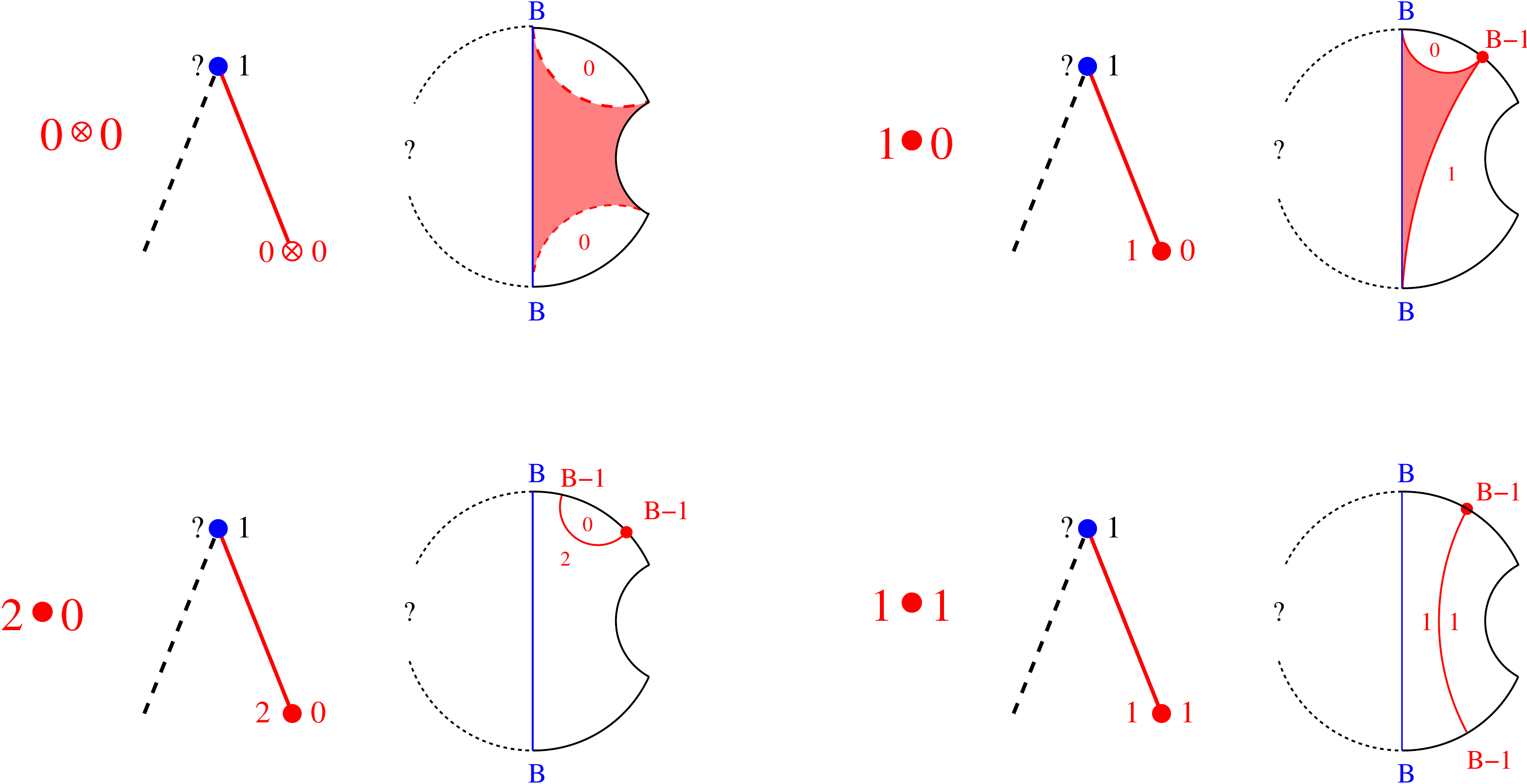}} 
  \caption{\small Possible nodes ending a branch with a label 1, i.e. child specifications
  of a periodic  1-specification with  two jumps.}\label{case1s}
\end{figure}

 \subsection{Childs of a periodic pseudo-orbit with 2 jumps.}\label{scase1}\quad
 
 \subsubsection{Case $0\otimes 0$. When both points in the  approach are the jumps of the pseudo-orbit. }
 Denote the mother $Q^{-B}$ pseudo-orbit by $(q_a,\ldots,q_{b-1},q_c,\dots,q_{d-1})$ with 2 jumps at
 $q_b=q_c$ and $q_d=q_a$. In this case the $\tfrac 12 Q^{-B+1}$ approach is $(q_i,q_j)=(q_a,q_c)$.
 Observe that 
 \begin{align}
 d(T(q_{b-1}),q_a) &\le d(T(q_{b-1}),q_c)+d(q_c,q_a)
 \le Q^{-B}+\tfrac 12 Q^{-B+1}< Q^{-B+1}.
 \label{0w0a}\\
 d(T(q_{d-1}),q_c)&\le d(T(q_{d-1}),q_a)+d(q_a,q_c)
 \le Q^{-B}+\tfrac 12 Q^{-B+1}< Q^{-B+1}.
 \label{0w0b}
 \end{align}
 By \eqref{0w0a} we have that $(q_a,\ldots,q_{b-1})$ is a periodic $Q^{-B+1}$ pseudo-orbit with only one jump.
 By \eqref{0w0b}, $(q_c,\ldots,q_{d-1})$ is another periodic $Q^{-B+1}$ pseudo-orbit with only one jump.
 The points in the approach $(x_a,x_c)$ may have both appeared before as nodes (i.e. approaches) in
 the tree, so we write a white (or empty) node $\otimes$. The label of the node is $0\otimes 0$ because both child 
 pseudo-orbits have only one jump. The node issues two branches corresponding to the numbers 0 and 0.
 In the disc we draw the lines $\ov{q_aq_b}$ and $\ov{q_cq_d}$ corresponding to the new approaches 
 and shadow the quadrilateral region limited by
 these lines and the previously drawn lines $\ov{q_bq_c}$ and $\ov{q_aq_d}$.
 
\subsubsection{Case $0\bullet 1$. When one point of the approach is one of the jumps of the pseudo-orbit.}\label{c101}
Denote the $Q^{-B}$ pseudo-orbit by $(q_a,\ldots,q_{b-1},q_c,\ldots, q_{d-1})$ with jumps at $q_b$ and $q_d$.
We can assume that in this case $q_a$ is one of the points in the $\tfrac 12 Q^{-B+1}$ 
approach $(q_i,q_j)=(q_a,q_j)$. We will further assume 
that $a<j< b$ as in Figure~\ref{case1s}, the other case $c<j<d$ is similar. 
The point $q_j$ has not appeared before in the tree, so we put a black node $\bullet$. 
We have that $(q_a,\ldots,q_{j-1})$ is a $Q^{-B+1}$ pseudo-orbit with only 1 jump, 
which gives a number $0\bullet$ in the tree. Observe that 
\begin{align*}
d(T(q_{d-1}),q_j) &\le d(T(q_{d-1}),q_a)+d(q_a,q_j) \\
&\le d(T(q_{d-1}),q_a)+d(q_i,q_j) 
\le Q^{-B}+\tfrac 12 Q^{-B+1}
<Q^{-B+1}.
\end{align*}
Therefore 
$(q_j,\ldots,q_{b-1},q_c,\ldots,q_{d-1})$ is a periodic $Q^{-B+1}$ pseudo-orbit with two
jumps. We write the label 1 in the node $0\bullet 1$.
The node issues two branches corresponding to the numbers 0 and 1.
In the disc we draw the lines $\ov{q_aq_j}$ and $\ov{q_jq_d}$. We shadow the triangular region bounded 
by the lines $\ov{q_dq_a}$, $\ov{q_aq_j}$ and $\ov{q_jq_d}$. We treat the shadowed region as a line
with right and left sides, at the choice of the reader. The white regions left by the shadow are the two
specifications with 1 and 2 jumps described above.

\subsubsection{Case $0\bullet 2$. When both of the points of the approach are in the interior of one segment of the pseudo-orbit.}
Let $(q_a,\ldots,q_{b-1},q_c,\ldots,q_{d-1})$ be the $Q^{-B}$ pseudo-orbit with jumps at $q_b$ and $q_d$. We can assume that
the approach $(q_i,q_j)$ is in interior of the first segment $(q_a,\ldots,q_{b-1})$ of the pseudo-orbit, i.e.
$a<i<j<b$. Both points of the approach did not appear before in the tree so this is a black node $\bullet$.
The segment $(q_i,\ldots,q_{j-1})$ is a periodic $Q^{-B+1}$ pseudo-orbit with only one jump, which gives a number 0 
in the node $0\bullet$. The rest of the pseudo-orbit is a periodic pseudo-orbit with 3 jumps: 
$(q_a,\ldots,a_{i-1},q_j,\ldots,q_{b-1},q_c,\ldots,q_{d-1})$. We write a number 2 in the node $0\bullet 2$.
We stop the process at the pseudo-orbit with 3 jumps. The node will issue only one branch corresponding to the number 0.

\subsubsection{Case $1\bullet 1$. When the points in the approach are in the interior of both segments of the pseudo-orbit.}
Let $(q_a,\ldots,q_{b-1},q_c,\ldots,q_{d-1})$ be the periodic $Q^{-B}$ pseudo-orbit. The indices of the approach $(x_i,x_j)$ satisfy
$a<i<b<c<j<d$. Both points of the approach did not appear before in the tree, so the node is black $\bullet$. 
Both $(q_a,\ldots,q_{i-1},q_j,\ldots,q_{d-1})$ and $(q_i,\ldots,q_{b-1},q_c,\ldots,q_{j-1})$ are 
periodic $Q^{-B+1}$ pseudo-orbits with 2 jumps,  thus the label of the node is $1\bullet 1$.
This node $1\bullet 1$ issues two branches, each one with the number 1.

\bigskip
\bigskip

      \begin{figure}[h]
     \resizebox*{14cm}{7.8cm}{\includegraphics{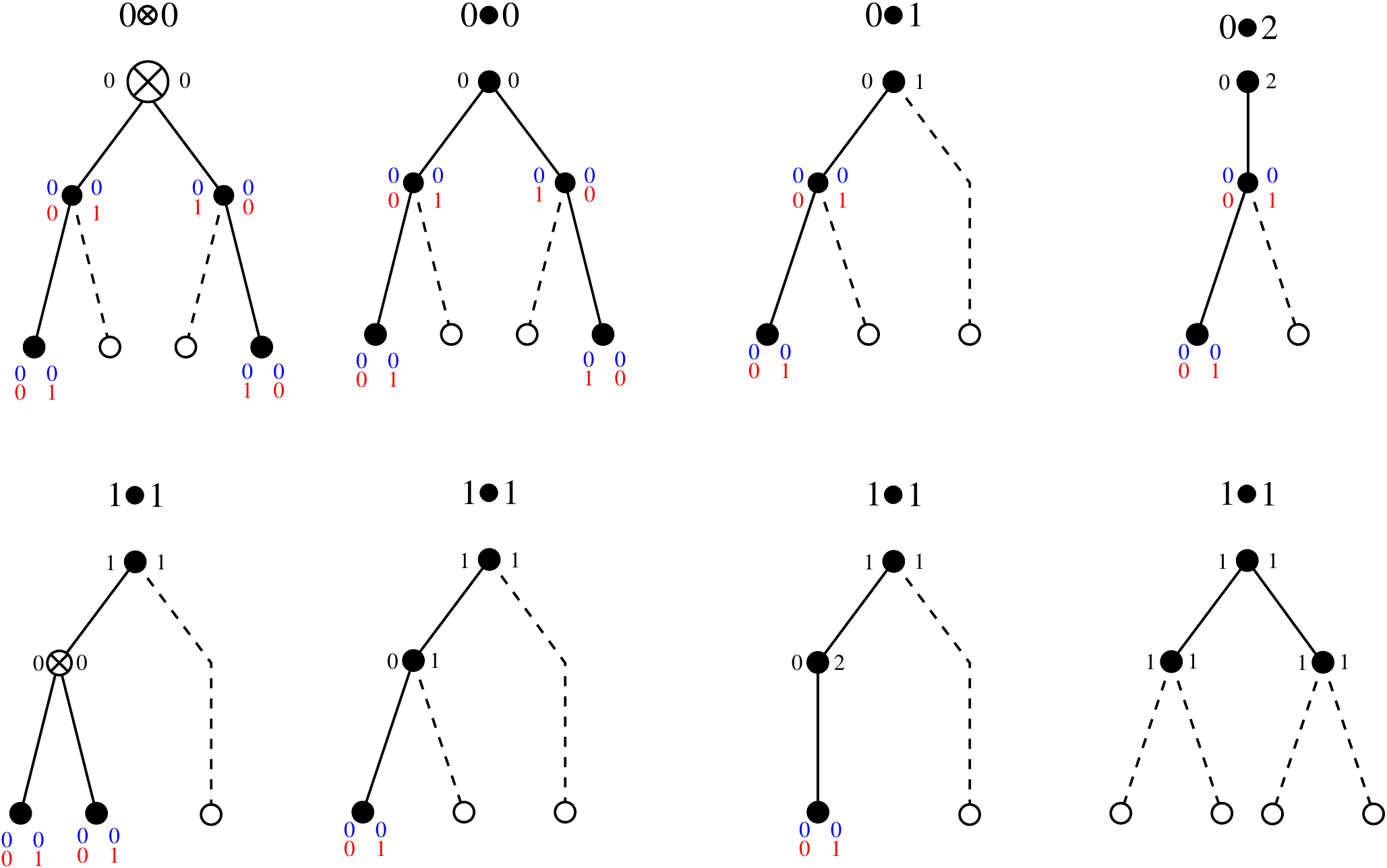}} 
  \caption{\small \openup 2pt
  This figure shows that the tree satisfies Claim~\ref{cl2step}. 
  The dots $\circ$ mean that we don't know if it is a white or black node.
  The dotted lines mean that we know that there is at least one branch, but we know
  neither the label of the branch nor the label and color of the ending node.
  The long dashed lines mean that the corresponding subtree has at least one ending node.
  The pictures use the fact from \S\ref{scase0} and Figure~\ref{case0} that a branch with label 0 can only end in nodes with labels 
  ${\color{blue} 0\bullet 0}$ or ${\color{red} 0\bullet 1}$. And in both of these cases the node has again at least 
  one new branch with label 0. All of the pictures satisfy Claim~\ref{cl2step}: i.e.
  at least two black dots in levels $N-1$, $N-2$ and at least two nodes, black or white, at the ending level $N-2$.
  }\label{2step}
\end{figure}

 The tree is built from the nodes described in \S\ref{scase0} and \S\ref{scase1} which also appear in
 Figure~\ref{case0} and Figure~\ref{case1s} respectively. In order to obtain the estimate in
  Proposition~\ref{expo} it is enough to show that at any consecutive pair of levels, the tree 
  duplicates its number of black nodes, because in that case we have
  $$
  t^N_{\ell+1}-t^N_\ell\ge\#\{\text{ black nodes }\}
  \ge 2^{\frac{N-N_0-1}2}.
  $$

  To obtain the duplication it is enough to show the following
  \begin{claim}\label{cl2step}
  At each node, black or white, in level $N>N_0+1$ the sub-tree below the node
  has at least two black nodes at levels $N-1$ and $N-2$  (added together) and also at 
  level $N-2$ the subtree of the node has at least two nodes, black or white.
  \end{claim}
  
  Because then at any two consecutive levels $N-1$, $N-2$ the number of black nodes  duplicates
  the number of nodes at level $N$  and also the total number of nodes at level $N-2$  duplicates 
  the number of nodes at level $N$. 
  
  In Figure~\ref{2step} we check that Claim~\ref{cl2step} is true.
  The figures take advantage (from \S\ref{scase0} and Figure~\ref{case0})
  that a branch with label 0 ends in a black node with label $0\bullet 0$ or $0\bullet 1$.
  In both cases the node has at least one branch with label 0 again.
  The dots $\circ$ mean that we don't know if the node is black $\bullet$ or white $\otimes$.
  In \S\ref{scase0} and \S\ref{scase1} (or Fig.~\ref{case0} and Fig.~\ref{case1s}) we see that
  all the labels for the nodes have at least one number smaller than 2. This implies that
  every node issues at least one branch. In Figure~\ref{2step} there are some long dashed 
  lines which mean that we know that there is at least one branch and at least one ending node,
  but we don't pay attention to more details.

The subtrees from a node $1\bullet 1$ are drawn in the lower line in Figure~\ref{2step}.
They are ordered by the first child node at the left hand side. 
The first three cases in the second row show that a subtree from a node $1\bullet 1$
which has a child with label either $0\otimes 0$, $0\bullet 1$ or $0\bullet 2$ satisfies the
Claim~\ref{cl2step}. The last case is a node $1\bullet 1$ with one left child  node $1\bullet 1$. 
 For the right branch, we have already seen that if the right node 
 is  $0\otimes 0$, $0\bullet 1$ or $0\bullet 2$ then the subtree satisfies Claim~\ref{cl2step}.
 It only remains the case in which the right node is also $1\bullet 1$.
 Figure~\ref{2step} shows that this last case also satisfies Claim~\ref{cl2step}.
 
 This completes the proof of Proposition~\ref{expo}. 
 
 \qed

  \appendix

 \color{black}
 

  \section{Zero Entropy.}  
  \label{aze}
  
  In this appendix we prove Ian Morris Theorem~\ref{Tmorris}.
  The published version was written for symbolic dynamics.
  We need two lemmas.
  
  \begin{Lemma}\label{A.1}
  Let $a_1,\ldots, a_n$ be non-negative real numbers, and let $A=\sum_{i=1}^n a_i\ge 0$.
  Then
  $$
  \sum_{i=1}^n -a_i\,\log a_i \le 1 + A\,\log n,
  $$
  where we use the convention $0\, \log 0 = 0$.
  \end{Lemma}
  
  \begin{proof}
  Applying Jensen's inequality to the concave function $x\mapsto -x \log x$ yields
  $$
  \frac 1n \sum_{i=1}^n-a_i\,\log a_i 
  \le -\left(\frac 1n \sum_{i=1}^n a_i\right)
  \log\left(\frac 1n \sum_{i=1}^n a_i\right)
  =-\frac An\,\log A + \frac An\, \log n
  $$
  from which the result follows.
 \end{proof}
 
 \medskip
 \begin{Lemma}\label{A.2}
 Let $f\in \Lip(X,\re)$ and suppose that $\cM_{\rm max}(f)=\{\mu\}$ for some $\mu\in \cM(T)$.
 Then there is $C>0$ such that for every $\nu\in\cM(T)$,
 $$
 -\a(f) - C\int d(x,K)\;d\nu \le \int f \;d\nu,
 $$
 where $K=\supp \mu$.
 \end{Lemma}
 
 \begin{proof}
 By Proposition~\ref{wK} and Lemma~\ref{LoF}.3.ii there exists $g\in \Lip(X,\re)$ such that 
 \linebreak
 $f+g-g\circ T\le -\a(f)$. Define $\tf = f +g - g\circ T $. 
  Since $\mu\in\cM_{max}(f)$, 
  $$
  \int \tf \, d\mu = \int f \, d\mu =-\a(f)
  \quad\text{ and }\quad
 \tf \le -\a(f).
 $$ 
 Since $\tf$ is continuous, it follows that $\tf(z)=-\a(f)$ for every $z\in K=\supp\mu$.
 Let $C=\Lip(\tf)$. Given $x\in X$, let $z\in K$ be such that $d(x,z)=d(x,K)$.
We have that
 $$
 \tf(x)\ge \tf(z)- C\,d(x,z)
 =-\a(f)-C\, d(x,K)
 $$
 from which the result follows.
 \end{proof}
  
  \medskip
  
     \noindent {\bf \ref{Tmorris}. Theorem }{(Morris \cite{Morris}){\bf .}
      ~\newline\indent{\it 
  Let $X$ be a compact metric space and $T:X\hookleftarrow$  an expanding map.
  There is a residual set $\cG\subset \Lip(X,\re)$ such that if
  $F\in \cG$ then there is a unique $F$-maximizing measure and 
  it has zero metric entropy.}
  
 \medskip
 
 \begin{Remark}\label{mergodic}
 By the linearity of the integral, or by the characterization of maximizing
 measures in Lemma~\ref{LoF}-2.(iii), the ergodic components of a maximizing 
 measure are also maximizing. Therefore the unique maximizing measure in
 Theorem~\ref{Tmorris} is ergodic. In fact  the map
 $T|_{\supp(\mu)}$  is uniquely ergodic.
 \end{Remark}

 \medskip

  \noindent{\bf Proof of Theorem~\ref{Tmorris}:}
  
  For $p\ge 1$ let $\cM^p(T)$ be the set of invariant probabilities supported on a 
  periodic orbit of period smaller or equal to $p$. In this appendix we will identify a periodic 
  orbit $\{z,Tz,\ldots,T^{p-1}z\}$ with the corresponding invariant measure 
  $\mu=\frac 1p \sum_{i=0}^{p-1} \delta_{T^iz}$.
  
  Let 
  \begin{equation}\label{e0}
  e_0>0, \qquad 0<\la<1
  \end{equation}
   be such that for every $x\in X$ the branches of the inverses of $T$ at $x$ are
  well defined, injective, and are $\la$-contractions on the ball 
  $B(x,\e_0)$ of radius $e_0$ centered at $x$.

    Let
   $$
   \cE_\ga:=\{\, f\in\Lip(X,\re)\;|\; h(\mu)< 2\,\ga\, h_{\rm top}(T) \quad \forall \mu\in\cM_{\rm max}(f)\,\}.
   $$
   By Theorem~\ref{CLT} the set 
   $$
   \cO=\{f\in\Lip(X,\re) \;|\; \#\cM_{\max}(f)=1\,\}
   $$
   is residual.
   
   It is enough to prove that $\cE_\ga$ is open and dense for every $\ga>0$,
   for then the set
   $$
   \cG=\cO\cap\bigcap_{n\in\na}\cE_{\frac 1n}
   $$ 
   satisfies the requirements of the Theorem.

   {\it Step 1.} $\cE_\ga$ is open.
   
   Suppose that $f\in \Lip(X,\re)$, $f_n\in \Lip(X,\re)\setminus \cE_\ga$ and $\lim_n f_n = f$ in $\Lip(X,\re)$.
   Then there are $\nu_n\in\cM_{\rm max}(f_n)$ with $h(\nu_n)\ge 2 \ga\, h_{\rm top}(T)$.
   Taking a subsequence if necessary, we may assume that $\nu_n\to \nu\in\cM(T)$.
   For any $\mu\in\cM(T)$ we have that
   $$
   \int f\, d\mu-\lV f -f_n\rV_\infty\le \int f_n\;d\mu
   \le \int f_n\;d\nu_n\le \int f\, d\nu_n +\lV f-f_n\rV_\infty.
   $$
   Taking $\lim_n$ we get that  $\int f\, d\mu \le \int f\, d\nu$ 
   for any $\mu\in\cM(T)$
   and hence $\nu\in \cM_{\rm max}(T)$.
   Since the map $m\mapsto h(m)$ is upper semicontinuous 
   (see e.g. Walters~\cite[Theorem 8.2]{Walters}) we have that
   $h(\nu) \ge 2\ga\,h_{\rm top}(T)$. Therefore $f\in \Lip(X,\re)\setminus \cE_\ga$. 
   We conclude that  $\Lip(X,\re)\setminus \cE_\ga$ is closed and then $\cE_\ga$ is open.

   {\it Step 2.} We have to prove that $\cE_\ga$ intersects every non-empty open set.
    Let $\cU\subset \Lip(X,\re)$ be open and non-empty. By Theorem~\ref{CLT} there is 
    $f\in \cU$ such that $\cM_{\rm max}(f)$ has only one element $\mu$. 
    If $\mu$ is a periodic orbit then $f\in\cE_\ga\cap \cU$ and we are done. 
    Otherwise, since by Lemma~\ref{LoF}-2.(iii) any measure in $\supp(\mu)$ 
    would also be maximizing,
    we have that $K:=\supp(\mu)$ does not contain a periodic orbit.
    By Lemma~\ref{A.2} there is
    a real number $C>0$ and a compact invariant set $K$ such that for every $\nu\in\cM(T)$ 
    \begin{equation} \label{Ae3}
    -\a(f)- C\int d(x,K)\; d\nu \le \int f\;d\nu
    \end{equation}
    and such that $K$ does not contain a periodic orbit.
    
    Let $\be>0$ be small enough that $f+g\in \cU$ whenever 
    \begin{equation}\label{Ubeta}
    \lV g\rV_{0}+\Lip(g)\le (\diam X+1)\be.
    \end{equation}
    We will construct a sequence of approximating functions such that
    $f_n\in \cU\cap \cE_\ga$
    for $n$ large enough.  In the next two steps we choose a sequence of periodic orbits which 
    will be used in the construction.
   
   {\it Step 3.}
   \begin{claim}\label{CTM} Given any $0<\theta<1$, 
   there is a sequence of integers $(m_n)_n$ and a sequence of periodic orbits
   $\mu_n\in\cM^n(T)$ such that
   \begin{gather*}
   \int d(x,K)\;d\mu_n =o(\th^{m_n})
   \qquad\text{ and }\qquad
   \lim_{n\to\infty}\frac{\log n}{m_n}=0.
   \end{gather*}
   \end{claim}
   
   \noindent{\it Proof of the Claim.}
   By a theorem of Bressaud and Quas \cite[Corollary 3 and Theorem 4]{BQ} for every $k>0$
   \begin{equation}\label{eBQ}
   \lim_{n\to+\infty}n^k\left( \inf_{\mu\in\cM^n(T)}\int d(x,K)\;d\mu\right) =0.
   \end{equation}
   Indeed  recall that using a Markov partition (cf. Ruelle~\cite[\S 7.29]{ruelle})
    the map $T$ is H\"older 
   continuously semi-conjugate to a subshift of finite type. 
   This is enough to obtain  estimate \eqref{eBQ}
   (see the proof of Corollary 3 in Bressaud and Quas \cite{BQ}).

   From~\eqref{eBQ} there exists a sequence of periodic orbits $\mu_n\in\cM^n(T)$ such that
   $$
   \lim_{n\to+\infty} n^k \int d(x,K)\; d\mu_n =0.
   $$
   Define
   $$
   r_n:=\log_\theta \left(\int d(x,K)\;d\mu_n\right).
   $$
   Since
   $$
   \theta^{r_n}\le n^k\,\theta^{r_n}\le 1
   \qquad\iff\qquad
   0\ge \frac{\log_\theta n}{r_n}\ge -\frac 1k,
   $$
   we have that $r_n^{-1}\log_\theta n \to 0$. Define $m_n:= \lfloor \frac 12 r_n\rfloor$,
   then $m_n^{-1} \log_\theta n \to 0$ and
   $$
   \int d(x,K)\;d\mu_n= \theta^{r_n}\le \theta^{m_n+\frac 12 r_n}=o(\theta^{m_n})
   $$
   as required.
   
   \bigskip
   
   {\it Step 4.} Using~\eqref{e0} fix 
      \begin{equation}\label{theta}
      0<\theta<\min\{e_0,\la, e_0 \Lip(T)^{-1}\}.
      \end{equation}
      Choose $m_n$ and $\mu_n$ as in Claim~\ref{CTM}.
      Define $L_n:=\supp \mu_n$.  
   \begin{claim}\label{CSA4}
   There is $N_\ga>0$ such that when $n\ge N_\ga$
   $$
   \nu(\{\,x\in X\;|\; d(x,L_n)\ge \th^{m_n}\,\})>\ga
   $$
   for every invariant measure $\nu\in\cM(T)$ such that
   $h(\nu)\ge 2 \ga\,h_{\rm top}(T)$.
   \end{claim}
   \noindent{\it Proof of the Claim.}

   Recall that a Markov partition for $T$ is a finite collection of sets $S_i$ 
   which cover $X$ such that
   \begin{enumerate}
   \renewcommand{\theenumi}{\alph{enumi}}
   \item $S_i = \ov{{\rm int}\, S_i}$.
   \item If $i\ne j$ then ${\rm int}\,S_i\cap {\rm int}\,S_j = \emptyset$.
   \item $f(S_i)$ is a union of sets $S_j$.
   \end{enumerate}
   Ruelle \cite[\S 7.29]{ruelle}
   proves that for expanding maps there are Markov partitions of arbitrarily small diameter.
   Let $\P$ be a Markov partition with $\diam\P<e_0$. The elements of the partition
   $$
   \P^{(n)}:=\bigvee_{i=0}^{n-1}T^{-i}\P 
   = \Big\{ \bigcap\limits_{i=0}^{n-1}A_i\;\Big|\; A_i\in T^{-i}\P\;\Big\}
   $$
   have diameter smaller than $\la^{n-1} e_0$ and contain an open set.
   Then the partition $\P$ is generating because the $\si$-algebra
     $$
   \P^\infty = \sigma\big(\cup_n \P^{(n)}\big)={\mathcal Borel}(X).
   $$
   contains all the open sets.\footnote{The star of a point $x$ in $\P^{(n)}$,
   $S(x)=\cup\{A\in\P^{(n)}\,|\, x\in A\,\}$, contains at most
   $\# \P$ elements, has diameter $\le 2\la^{n-1} e_0$ and contains a neighborhood 
   of the point $x$. Therefore any open set in $X$ is a union of (countably many) elements
   of $\cup_n \P^{(n)}$. }
    Therefore (cf. Walters~\cite[Thm.~4.18]{Walters})
    for every invariant measure $\nu\in\cM(T)$,
   $$
   h(\nu) = \inf_k \frac 1k \sum_{A\in\P^{(k)}}-\nu(A)\,\log\nu(A).
   $$
      From the definition of topological entropy using covers (cf. Walters~\cite[\S 7.1]{Walters}) we have that
   $$
   \lim_{k\ge 1}\frac 1k \,\log\#\, \P^{(k)} \le h_{\rm top}(T).
   $$
   Choose $N_\ga$ large enough such that for all $n\ge N_\ga$
   \begin{gather}
   \frac {2+\log\#\P}{m_n} +\frac{\log n}{m_n} +\frac\ga{m_n}\log\#\P^{(m_n)} < 2\ga\, h_{\rm top}(T).
   \label{20}
   \end{gather}
   
     Let $\nu\in\cM(T)$ and suppose that
   \begin{equation}\label{lega}
   \nu(\{x\in X\;|\; d(x,L_n)\ge \th^{m_n}\})\le \ga
   \end{equation}
   for some $n\ge N_\ga$.
   We will show that necessarily $h(\nu)<2\ga\,h_{\rm top}(\nu)$.
  
   Let
   $$
   W_n:= \{\, A\in \P^{(m_n)} \;|\; d(x,L_n)<\th^{m_n} \quad \text{\rm for some } x\in A\,\}.
   $$
   From \eqref{lega},
   $$
   \tga_n :=\sum_{A\in\P^{(m_n)}\setminus W_n}\nu(A) \le \ga.
   $$
  Using lemma~\ref{A.1} we have that
  \begin{align}
  h(\nu) &\le \frac 1{m_n}\sum_{A\in W_n}-\nu(A)\,\log\nu(A)
  +\frac 1{m_n}\sum_{A\in\P^{(m_n)}\setminus W_n}-\nu(A)\,\log\nu(A)
  \notag \\
  &\le \frac 1{m_n}\big(1+(1-\tga_n)\log\# W_n\big)
  +\frac 1{m_n}\big( 1+\ga\,\log\# \P^{(m_n)}\big).
  \label{hnu}
  \end{align}
  Let $g$ be a branch of the inverse of $T^{m_n}$. If $x,y$ are 
  in the domain of $g$, we have that
 $$
  d(g(x),g(y)) \ge {\Lip(T)^{-m_n}} d\big(T^{m_n}(g(x)),T^{m_n}(g(y))\big)
  \ge \Lip(T)^{-m_n} d(x,y).
  $$
      Using \eqref{theta},
  observe that since $\th^{m_n}<e_0\,\Lip(T)^{-m_n}$ for any $y\in L_n$ 
  there is a branch $g$ of the inverse of $T^{m_n}$ such that the ball
  $$
  B(y,\th^{m_n}) \subseteq g\big(B(T^{m_n}y,e_0)\big).
  $$
  Since $\P$ is a Markov partition with $\diam \P<e_0$, 
  $$
  \P^{(m_n)}=\{\, g(A)\;|\; A\in\P,\quad g\text{ is branch of $T^{-m_n}$}\,\}.
  $$
  Therefore the ball $B(y, \th^{m_n})$ intersects 
  at most $\# \P$ elements of $\P^{(m_n)}$ because
  by applying $T^{m_n}$
  $$
  \#\{\,B\in\P^{(m_n)}\;|\; B\cap B(y,\th^{m_n})\ne \emptyset\}
  \le 
  \#\{\, A\in\P\;|\; A\cap B(T^{m_n}y,e_0)\ne \emptyset\,\}
  \le \# \P.
  $$
  Since $L_n$ has at most $n$ elements, $\#W_n\le n\,\#\P$.
  Thus from \eqref{hnu} and  \eqref{20} we have that
  \begin{align*}
  h(\nu) &\le 
  \frac 1{m_n}\big(1+(1-\tga_n)\log n \, \# \P\big)
  +\frac 1{m_n}\big( 1+\ga\,\log\#\P^{(m_n)}\big).
  \\
  &\le \frac {2+\log\#\P}{m_n} +\frac{\log n}{m_n} + \frac\ga{m_n}{\log\#\P^{(m_n)}}
  < 2\ga\,h_{\rm top}(T).
   \end{align*}
   This proves the claim.

   \noindent{\it Step 5.} We now complete the proof. Define a sequence of functions 
   $f_n\in\Lip(X,\re)$ by
   \begin{equation}\label{fn}
   f_n(x)=f(x)-\be\,d(x,L_n),
   \end{equation}
   where $L_n=\supp\,\mu_n$ as above. From the definition of $\be$ in~\eqref{Ubeta}
   we have that
   $f_n\in \cU$ for each $n\ge 1$.  From Claim~\ref{CTM} in step 3 we have that
   \begin{equation*}\label{omn}
   \int d(x,K) \; d\mu_n = o(\th^{m_n}),
   \end{equation*}
   and from Claim~\ref{CSA4} in step 4 it follows that when $n$ is sufficiently large,
   $$
   \int d(x,L_n)\;d\nu \ge \th^{m_n}\,\nu\big(\{x\in X\;|\;d(x,L_n)\ge \th^{m_n}\}\big)
   \ge \ga\, \th^{m_n}
      $$
   for all $\nu\in \cM(T)$ such that $h(\nu)\ge 2\ga\,h_{\rm top}(T)$. 
   
   We may therefore choose $n$ such that $\be \int d(x,L_n)\,d\nu > C \int d(x,K)\,d\mu_n$
   for every $\nu\in \cM(T)$ such that $h(\nu)\ge 2\ga\, h_{\rm top}(T)$.
   It follows that for every such measure $\nu$
   \begin{align*}
   \int f_n\;d\nu 
   &= \int f\; d\nu -\be \int d(x,L_n)\;d\nu \\
   &<-\a(f)-C\int d(x,K)\;d\mu_n \\
   &\le \int f \; d\mu_n 
   =\int f_n\;d\mu_n
   \le -\a(f_n),
   \end{align*}
   where we have applied \eqref{Ae3} and \eqref{fn}.
   We have shown that if $\nu\in \cM(T)$ and $h(\nu)\ge 2\ga\, h_{\rm top}(T)$,
   then $\nu\notin \cM_{\rm max}(f_n)$, and therefore  $f_n\in\cE_\ga\cap U$.
   We conclude that $\cE_\ga$ is dense in $\Lip(X,\re)$ and the theorem is proved.
   
   \qed


  \def\cprime{$'$} \def\cprime{$'$} \def\cprime{$'$} \def\cprime{$'$}
\providecommand{\bysame}{\leavevmode\hbox to3em{\hrulefill}\thinspace}
\providecommand{\MR}{\relax\ifhmode\unskip\space\fi MR }
\providecommand{\MRhref}[2]{%
  \href{http://www.ams.org/mathscinet-getitem?mr=#1}{#2}
}
\providecommand{\href}[2]{#2}

\end{document}